\documentclass[11pt,reqno,twoside]{amsart}
\numberwithin{equation}{section}

\usepackage{xypic} 
\usepackage{amssymb}
\usepackage{mathrsfs}
\usepackage{hyperref}
\usepackage{comment}
\xyoption{curve}

\hypersetup{ 
colorlinks,
citecolor=black,
filecolor=blue,
linkcolor=black,
urlcolor=blue
}

\newtheorem{lemma}{Lemma}[section]

\newtheorem{proposition}[lemma]{Proposition}
\newtheorem{theorem}[lemma]{Theorem}
\newtheorem{claim}[lemma]{Claim}
\newtheorem{corollary}[lemma]{Corollary}

\theoremstyle{definition}

\newtheorem{definition}[lemma]{Definition}
\newtheorem{remark}[lemma]{Remark}

\newtheorem{notation}[lemma]{Notation}

\newcommand{\ba}{\begin{array}}
\newcommand{\ea}{\end{array}}
\newcommand{\Hom}{\mathrm{Hom}}

\newcommand{\Ext}{\mathrm{Ext}}

\newcommand{\End}{\mathrm{End}}
\newcommand{\ox}{\otimes}

\newcommand{\xym}{\xymatrix}
\newcommand{\uHom}{\underline{\Hom}}
\newcommand{\G}{E}

\newcommand{\Wedge}{\textstyle\bigwedge}
\newcommand{\unl}{\underline}

\title[Braided Hochschild cohomology]{Braided Hochschild cohomology and Hopf actions}
\date{}

\author{Cris Negron}
\address{Department of Mathematics\\Louisiana State University\\
Baton Rouge, LA 70803, USA}
\email{cnegron@lsu.edu}
\thanks{This work was supported by NSF Postdoctoral fellowship DMS-1503147}

\begin{document}
\maketitle

\begin{abstract}
We show that the braided Hochschild cohomology, of an algebra in a suitably algebraic braided monoidal category, admits a graded ring structure under which it is braided commutative.  We then give a canonical identification between the usual Hochschild cohomology ring of a smash product and the (derived) invariants of its braided Hochschild cohomology ring.  We apply our results to identify the associative formal deformation theory of a smash product with its formal deformation theory as a module algebra over the given Hopf algebra (when the Hopf algebra is sufficiently semisimple).  As a second application we deduce some structural results for the usual Hochschild cohomology of a smash product, and discuss specific implications for finite group actions on smooth affine schemes.
\end{abstract}

\section{Introduction}

Let $k$ be a field and $\mathscr{Z}$ be the category of modules over a {\it quasitriangular} Hopf algebra $H$ (see \ref{sect:quasitriangle}).  This gives $\mathscr{Z}$ the structure of a braided monoidal category with a monoidal embedding into the category of vector spaces.  Our motivating example is the category $YD^E_E$ of Yetter-Drinfeld modules over a finite dimensional Hopf algebra $E$.
\par

An algebra in $\mathscr{Z}$ will be an object $B$ in $\mathscr{Z}$ equipped with a compatible algebra structure.  That is, we require all the structure maps
\[
\mathrm{unit}:k\to B\ \ \text{and}\ \ \mathrm{mult}:B\ox B\to B
\]
to be maps in $\mathscr{Z}$.  This is equivalent to $B$ being a $H$-module algebra when we specify $\mathscr{Z}=H\mathrm{mod}$.  Our primary example of such an algebra is a smash product $B=A\ast E$, where $E$ is a finite dimensional Hopf algebra acting on $A$.  The smash product is an algebra in the category of Yetter-Drinfeld modules over $E$ under the right adjoint action (see \ref{sect:smashprud}).  
\par

The work of the present paper is motivated, in part, by a desire to understand the nature of the Hochschild cohomology of $A$ with coefficients in the smash product $HH^\bullet(A,A\ast E)$.  This cohomology is, in a sense, more fundamental than the usual Hochschild cohomology $HH^\bullet(A\ast E)$.  This can be seen in many works on Hochschild, cyclic, and orbifold cohomology (e.g. ~\cite{baranovsky03,dolgushevetingof05,SW,SW2,halbouttang10,pflaumetal11,brodzkietal15}).  As is explained below, we find that the cohomology $HH^\bullet(A,A\ast G)$ is actually the {\it braided} Hochschild cohomology of the smash product, with coefficients in itself.  This phenomenon was first noticed by Schedler and Witherspoon in a more restrictive setting~\cite{schedlerwitherspoon}.

In~\cite{baez94} Baez considered the braided Hochschild homology of such an algebra $B$.  He defined this homology as the usual Tor group of the corresponding {\it braided enveloping algebra} $B^{\unl{e}}$,
\[
H^c_\bullet(B)=\mathrm{Tor}^{\mathrm{B}^{\unl{e}}}_\bullet(B,B).
\]
Here $c$ denotes the braiding transformation on $\mathscr{Z}$.  (One can see Section \ref{sect:braidedhh} for a precise definition of the braided enveloping algebra.)  Baez related this homology to $c$-commutative differential forms on $B$, when $B$ is braided commutative.  Baez's results are referential to the usual Hochschild homology of a commutative algebra, in which case one finds Kh\"aler differentials in degree $1$.
\par

In this work we consider the braided Hochschild {\it co}homology, which can be defined succinctly as the group of self extensions of $B$ over its braided enveloping algebra.  We follow Baez's notation and take
\[
H^\bullet_c(B):=\Ext^\bullet_{B^{\unl{e}}}(B,B).
\]
We also consider a relative version of the cohomology $H^\bullet_{c,E}(B)$, which is defined for any subalgebra $E\subset B$ in $\mathscr{Z}$.
\par

This work is concerned primarily with a graded ring structure on braided Hochschild cohomology and its relation(s) to the non-braided Hochschild cohomology of smash products, although there are also interactions between the braided Hochschild cohomology and the Gerstenhaber bracket~\cite[Ch. 3]{negronthesis}.

\begin{theorem}[=\ref{thm:braided_comm}]\label{thm:3}
Let $B$ be an algebra in a braided monoidal category $\mathscr{Z}$ as above.
\begin{enumerate}
\item Each cohomology group $H^i_c(B)$ is itself an object in $\mathscr{Z}$.
\item The cohomology $H^\bullet_c(B)$ admits a graded ring structure under which it is braided commutative.
\end{enumerate}
The two above statements also hold for any relative cohomology $H^\bullet_{c,E}(B)$.
\end{theorem}

Theorem \ref{thm:3} is a quantum analog of the well known fact that the usual Hochschild cohomology of any algebra is a (graded) commutative ring~\cite{gerstenhaber63}.  We also prove

\begin{theorem}[=\ref{cor:HcVHH}/\ref{cor:749}]\label{thm:01}
Let $E$ be a finite dimensional Hopf algebra and $A$ be an $E$-module algebra.  There is an algebra isomorphism between the Hochschild cohomology $HH^\bullet(A,A\ast E)$ and the $E$-relative braided Hochschild cohomology $H^\bullet_{c,E}(A\ast E)$.  When, further, $E$ is semisimple and cosemisimple the Hochschild cohomology $HH^\bullet(A,A\ast E)$ is isomorphic to the non-relative braided cohomology $H^\bullet_{c}(A\ast E)$.
\end{theorem}

Theorem \ref{thm:01} was first proved by Schedler and Witherspoon in the case of a finite group acting on an algebra in characteristic $0$~\cite{schedlerwitherspoon}.  We also give a version of Theorem \ref{thm:01} for certain crossed products and {\it twists} of smash products (see \ref{sect:w_cocycle}).  Theorem \ref{thm:01} implies that there is an algebra identification between $HH^\bullet(A\ast E)$ and the $E$-invariant classes in $H_c^\bullet(A\ast E)$ when $E$ is semisimple and cosemisimple, among other things (see Corollary \ref{cor:749}).  The theorem also implies that the cohomology $HH^\bullet(A,A\ast E)$ takes values in the category of Yetter-Drinfeld modules over $E$, and that it is braided commutative in that category.
\par

The fact that, when $E$ is semisimple and cosemisimple, we have an isomorphism with the non-relative cohomology $H^\bullet_c(A\ast E)$ is of some significance.  For example, we can use this fact to show that the formal deformation theory of $A\ast E$ as an associate algebra is equivalent to the formal deformation theory of $A\ast E$ as an $E$-module algebra (under the adjoint action).  This is done in Section \ref{sect:defthry}.
\par

We also apply the two above theorems to gain some general information about the Hochschild cohomology of smash products $A\ast G$, where $G$ is a finite group.  As mentioned above, in characteristic $0$ there is an algebra equality
\[
HH^\bullet(A\ast G)=HH^\bullet(A,A\ast G)^G\ \ (\text{and now }HH^\bullet(A\ast G)=H_c^\bullet(A\ast G)^G),
\]
and more generally a spectral sequence relating the two cohomologies.  Whence we study the cohomology $HH^\bullet(A\ast G)$ by studying $HH^\bullet(A,A\ast G)$.
\par

The cohomology $HH^\bullet(A,A\ast G)$ decomposes as a sum
\[
HH^\bullet(A,A\ast G)=\bigoplus_{g\in G} HH^\bullet(A,Ag),
\]
and we get an action of the Hochschild cohomology $HH^\bullet(A)=HH^\bullet(A,Ae)$ on each summand induced by the multiplication on $HH^\bullet(A,A\ast G)$.  In certain examples, these $HH^\bullet(A)$-actions determine a great deal of the multiplicative structure on the entire cohomology~\cite{SW,pflaumetal11}.
\par

In Proposition \ref{prop:whatever!}, we apply Theorems \ref{thm:3} and \ref{thm:01} to show that each summand $HH^\bullet(A,Ag)$ is a symmetric bimodule over the Hochschild cohomology of $A$, and find explicit ideals in $HH^\bullet(A)$ annihilating each given summand.  In the case of a finite group acting on a smooth affine $k$-scheme $X$, for example, our results imply that each cohomology $HH^\bullet(k[X],k[X]g)$ carries a canonical module structure over the algebra of polyvector fields $\wedge^\bullet T_{X^g}$ on the fixed space $X^g$ (see \ref{sect:ideals}).  This result is expected from its $C^\infty$-analog, which was studied in \cite{pflaumetal11}.  We expect that the ideals specified in Proposition \ref{prop:whatever!} are in fact the entire annihilators of each $HH^\bullet(k[X],k[X]g)$ for a general smooth affine $G$-scheme $X$ (see again \ref{sect:ideals}).

\setcounter{tocdepth}{1}
\tableofcontents

\section{Background}

\subsection{Conventions}
Throughout this work $\mathscr{Z}$ will always denote a braided monoidal category which admits an equality $\mathscr{Z}=H\mathrm{mod}$, for some quasitriangular Hopf algebra $H$.  We also fix the Hopf algebra $H$ throughout.  The invariants $(-)^\mathscr{Z}$ is defined abstractly as maps from the unit $(-)^{\mathscr{Z}}=\Hom_\mathscr{Z}(\mathbf{1}_\mathscr{Z},-)$.  This will agree with the usual invariants $(-)^\mathscr{Z}=(-)^H$.  We let $c$ denote the braiding transformation on $\mathscr{Z}$.  
\par

A dot $\cdot$ generally denotes the action of $H$, while juxtaposition denotes multiplication in a ring.  All Hopf algebras are assumed to have bijective antipode.
\par

By $\ox$ we mean $\ox_k$.  We use a simplified Sweedler's notation $\Delta(x)=x_1\ox x_2$ to denote comultiplication in a coalgebra.  So ``$x_1\ox x_2$" here is a symbol denoting a sum of pure tensors.
\par

Given an abelian category $\mathscr{M}$, we let $\mathrm{dg}\mathscr{M}$ denote the category of cochain complexes of objects in $\mathscr{M}$.  For any two complexes $M$ and $N$ in $\mathrm{dg}\mathscr{M}$ we always let $\Hom_{\mathscr{M}}(M,N)$ denote the hom {\it complex}
\[
\Hom_{\mathscr{M}}(M,N)=\bigoplus_{n\in\mathbb{Z}}\left\{\ba{c}\text{homogenous degree $n$ maps}\\
M\to N\text{ in }\mathscr{M}\ea\right\}.
\]
The differential is the usual one 
\[
d_{\Hom_{\mathscr{M}}(M,N)}(f)=d_Nf-(-1)^{|f|}fd_M.
\]
We view a ``graded object" in $\mathscr{M}$ as a complex in $\mathrm{dg}\mathscr{M}$ with vanishing differential.  In particular, the cohomology $H^\bullet(M)$ of a complex in $\mathrm{dg}\mathscr{M}$ will be seen as an object in $\mathrm{dg}\mathscr{M}$ with vanishing differential.

\subsection{Quasitriangular Hopf algebras}
\label{sect:quasitriangle}

We follow, for the most part, the notation of~\cite{EGNO}.  For a quasitriangular Hopf algebra $H$, with $R$-matrix $R\in H\ox H$, we take $R=\sum_j r_j\ox r^j$ and $R_{21}=\sum_j r^j\ox r_j$.  We adopt an Einstein sum notation and write $r_j\ox r^j$ and $r^j\ox r_j$ for these elements respectively.  The braiding on the corresponding category $\mathscr{Z}=H\mathrm{mod}$ will be given by
\begin{equation}\label{eq:86}
c_{M,N}(m\ox n)=(r^j \cdot n)\ox (r_j\cdot m).
\end{equation}
The $R$-matrix satisfies a number of relations, including the following:
\begin{equation}\label{eq:relations}
\ba{l}
\bullet\ \Delta(h)R_{21}=R_{21}\Delta^{op}(h),\\
\bullet\ (id\ox S)(R_{21})=(S^{-1}\ox id)(R_{21})=(R_{21})^{-1},\\
\bullet\ (\epsilon\ox id)(R_{21})=(id\ox\epsilon)(R_{21})=1\ox 1.
\ea
\end{equation}
\begin{equation}\label{eq:braidrels}
\ba{l}
\bullet\ r^j\ox (r_j)_1\ox (r_j)_2=r^lr^j\ox r_l\ox r_j,\hspace{1.7cm}\\
\bullet\ (r^j)_1\ox (r^j)_2\ox r_j=r^j\ox r^l\ox r_lr_j.
\ea
\end{equation}
The latter two relations are the {\it braid relations}, and all of the relations can be equated with the fact that $c$, as defined at (\ref{eq:86}), is a braiding on the category $H\mathrm{mod}$.

\subsection{The algebra $A\ast\G$ and $YD^E_E$}
\label{sect:smashprud}

We recall that the category of (right) Yetter-Drinfeld modules over a Hopf algebra $E$ is the category whose objects are vector spaces $M$ equipped with both a right $E$-action and right $E$-coaction which satisfy the compatibility
\[
(m\cdot w)_0\ox (m\cdot w)_1=m_0\cdot w_2\ox S(w_1)m_1w_3,
\]
for each $m\in M$ and $w\in E$.  Morphisms in $YD^E_E$ are maps $M\to N$ which are simultaneously module and comodule maps.  This category is braided monoidal, with braiding
\[
c_{M,N}:M\ox N\to N\ox M,\ \ m\ox n\mapsto n_0\ox (m\cdot n_1).
\]
When $E$ is finite dimensional, the category $YD^E_E$ is equal to the category of modules over a quasitriangular Hopf algebra called the Drinfeld, or quantum, double $D$ of $E$.  (Or rather, the double of the opposite algebra of $E$ if we follow~\cite{EGNO}).  So $YD^E_E$ can serve as one of our categories $\mathscr{Z}$.
\par

The Hopf algebra $D$ will contain $E^{op}$ and $E^\ast$ as Hopf subalgebras, and the restriction of the multiplication map $E^{op}\ox E^\ast\to D$ will be a vector space isomorphism.  Here $E^{op}$ means $E$ with the opposite multiplication, usual comultiplication, and inverted antipode $S_{E^{op}}=S^{-1}_E$.  If we take a basis $\{r_j\}_j$ of $E$, and dual basis $\{r^j\}_j$ of $E^\ast$, then the $R$-matrix in $D$ will be the element $\sum_jr_j\ox r^j$.
\par

Let $A$ be a $E$-module algebra.  We denote the given action of $E$ on $A$ by a superscript, ${^wa}:=w\cdot a$, and denote elements in the smash product $A\ast E$ by juxtaposition $aw=a\ast w$.  Any such smash product $A\ast\G$ becomes an algebra object in $YD^E_E$ under the right adjoint action 
\[
(aw)\cdot w':=S(w'_1) aww'_2=({^{S(w'_2)}a})S(w'_1)ww'_3
\]
and obvious coaction $\rho(aw)=aw_1\ox w_2$.  Or, if we think about $E^{op}$ in the Drinfeld double $D$, we have the {\it left} action given by
\[
w'\cdot (aw)=S_E(w'_1) aww'_2
\]
In the above expressions $a\in A$ and $w,w'\in E$.
\par

Throughout this work a smash product $A\ast E$ we be considered as an algebra object in the braided monoidal category $YD^E_E$.

\subsection{Inner homs}\label{sect:inhoms}

Given any Hopf algebra $H$, the category $H\mathrm{mod}$ admits inner homs, i.e. a functor $\uHom:H\mathrm{mod}^{op}\times H\mathrm{mod}\to H\mathrm{mod}$ with an adjunction
\begin{equation}\label{eq:adj}
\Hom_{H\mathrm{mod}}(L\ox M,N)\cong \Hom_{H\mathrm{mod}}(L,\uHom(M,N)).
\end{equation}
As a vector space we will have $\uHom(M,N)=\Hom_k(M,N)$, and the adjunction (\ref{eq:adj}) is the usual $\ox\text{-}\Hom$ adjunction.  The action of $H$ on $\uHom(M,N)=\Hom_k(M,N)$ will be given by $h\cdot f=\big(m\mapsto h_1\cdot f(S(h_2)\cdot m\big)$, for any $h\in H$, $f\in \Hom_k(M,N)$, $m\in M$.  One can check easily the identity
\[
\uHom(M,N)^H=\Hom_{H\mathrm{mod}}(M,N).
\]
Under the aforementioned action the canonical pairing
\[
\uHom(M,N)\ox M\to N
\]
becomes $H$-linear.  Indeed, this is the {\it unique} action so that the pairing is $H$-linear.

\section{Braided Hochschild cohomology as a $\mathscr{Z}$-valued invariant}
\label{sect:braidedhh}

Take a braided monoidal category $\mathscr{Z}=H\mathrm{mod}$ and algebra $B$ in $\mathscr{Z}$.  In this section we cover all of the basics for the braided Hochschild cohomology $H^\bullet_c(B)$.  We show that the braided cohomology can be identified with the cohomology of a certain canonical complex, which is a braided analog of the standard Hochschild cochain complex.  This complex will be a complex in $\mathscr{Z}$, so that each cohomology group $H^i_c(B)$ will be seen to be an object in $\mathscr{Z}$.  
\par

We then go on to consider a relative version of the braided Hochschild cohomology.  The relative cohomology will provide an essential tool in establishing relations between the braided Hochschild cohomology and standard, non-braided, Hochschild cohomology.
\par

We first recall some facts from the original work~\cite{baez94}.  Given an algebra $B$ in $\mathscr{Z}$ we let $B^{\unl{op}}$ denote the braided opposite algebra, which has multiplication $b\cdot_{\unl{op}}b':=(r^j\cdot b')(r_j\cdot b)$.  Given algebras $B$ and $C$ in $\mathscr{Z}$ we let $B\unl{\ox} C$ denote the braided tensor algebra.  This is the vector space $B\ox C$ with the multiplication
\[
(b\ox c)(b'\ox c')=\big(b(r^j\cdot b')\big)\ox \big((r_j\cdot c)c'\big).
\]
We let $B^{\unl{e}}$ denote the enveloping algebra $B^{\unl{op}}\unl{\ox} B$.  The algebra $B$ becomes a {\it right} module over $B^{\unl{e}}$ under the action
\[
B\ox (B^{\unl{op}}\unl{\ox} B)\to B,\ \ a\ox (b\ox b')\mapsto (r^j\cdot b)(r_j\cdot a)b'.
\]
That is to say, we follow the obvious sequence
\[
\xym{
(B)\ox B\ox B\ar[r]^{c_{B,B}\ox id} & B\ox (B)\ox B\ar[r]^(.7){m^{(2)}} & B, 
}
\]
where $m^{(2)}$ is the second iteration of the multiplication on $B$.

\begin{definition}\label{def:ogdef}
The braided Hochschild cohomology is defined as the extension group $H^\bullet_c(B):=\Ext^\bullet_{\mathrm{mod}B^{\unl{e}}}(B,B)$
\end{definition}

It is shown in~\cite{baez94} that the standard bar resolution
\[
\cdots \to B\ox B^{\ox 2}\ox B\to B\ox B\ox B\to B\ox B\to 0
\]
becomes a projective resolution of $B$ over $B^{\unl{e}}$ when we endow each $B\ox B^{\ox n}\ox B$ with the $B^{\unl{e}}$-action
\begin{equation}\label{eq:actn}
(a\ox x\ox a')(b\ox b'):=(r^j\cdot b)\big(r_j\cdot (a\ox x\ox a')\big)b',
\end{equation}
where $x\in B^{\ox n}$ and $b(a\ox x\ox a')b'=(ba)\ox x\ox (a'b')$.  (Here $H$ acts diagonally on the higher tensor powers of $B$.)  We denote this complex by $\mathscr{B}ar^cB$.
\par

To see that $\mathscr{B}ar^c B$ is projective in each degree we note that the $B^{\unl{e}}$-action map restricts to give an isomorphism $B^{\ox n}\ox B^{\unl{e}}\to B\ox B^{\ox n}\ox B$, for each $n$.  The inverse to this map is given by
\begin{equation}\label{eq:94}
\ba{l}
B\ox B^{\ox n}\ox B\to B^{\ox n}\ox B^{\unl{e}},\\
b\ox x\ox b'\mapsto (S(r_j)\cdot x)\ox (r^j\cdot b)\ox b'=(r_j\cdot x)\ox (S^{-1}(r^j)\cdot b)\ox b'
\ea
\end{equation}
\par

Recall that the category $\mathscr{Z}=H\mathrm{mod}$ has inner homs.  These objects are denoted $\uHom$ and are equal to the standard $k$-linear homs $\Hom_k$ along with the left $H$-action as described in Section \ref{sect:inhoms}.

\begin{lemma}\label{lem:105}
For each $n$ we have an isomorphism 
\[
\Hom_{\mathrm{mod}B^{\unl{e}}}(B\ox B^{\ox n}\ox B, B)\overset{\cong}\to \uHom(B^{\ox n},B)
\]
given by restriction.  The inverse of the restriction map is given by
\begin{equation}\label{eq:og_rels}
\ba{c}
\uHom(B^{\ox n},B)\to \Hom_{\mathrm{mod}B^{\unl{e}}}(B\ox B^{\ox n}\ox B, B)\\
f\mapsto \Big(b\ox x\ox b'\mapsto (r^j\cdot b) \big(r_j\cdot f\big)(x)b'\Big).
\ea
\end{equation}
\end{lemma}

\begin{proof}
First note that restriction does in fact provide an isomorphism between $\Hom_{\mathrm{mod}B^{\unl{e}}}(B\ox B^{\ox n}\ox B, B)$ and $\uHom(B^{\ox n},B)$, since each $B\ox B^{\ox n}\ox B$ is free over $B^{\unl{e}}$.  Now, we have for any $F\in \Hom_{\mathrm{mod}B^{\unl{e}}}(B\ox B^{\ox n}\ox B,B)$
\[
\begin{array}{ll}
F(b\ox x\ox b')&=F(1\ox S(r_j)\cdot x\ox 1)\big((r^j\cdot b)\ox b'\big)\\
&=\big(r^lr^j\cdot b)\big(r_l\cdot F(1\ox S(r_j)\cdot x\ox 1))b'.
\ea
\]
The braid relation
\[
(id\ox \Delta)(R_{21})=r^lr^j\ox r_l\ox r_j,
\]
then implies
\[
\big(r^lr^j\cdot b)\big(r_l\cdot F(1\ox S(r_j)\cdot x\ox 1)\big)b'=(r^j\cdot b)\Big((r_j)_1\cdot F(1\ox S\big((r_j)_2\big)\cdot x\ox 1)\Big)b'
\]
If we let $f=F| B^{\ox n}$ the final expression is equal to $(r^j\cdot b_0) \big(r_j\cdot f\big)(x)b'$.  This shows that the preimage of $f$ along the restriction isomorphism, which in this case is $F$, is given by the proposed formula (\ref{eq:og_rels}).
\end{proof}

As a result of the lemma we get an explicit isomorphism of graded spaces
\begin{equation}\label{eq:196}
\Hom_{\mathrm{mod}B^{\unl{e}}}(\mathscr{B}ar^c B, B)\overset{\cong}\to \bigoplus_{n\geq 0} \uHom(B^{\ox n},B).
\end{equation}
This induces a differential on the right hand side under which the two complexes will be canonically isomorphic.
\par

In the following lemma $d_{\mathscr{B}B}$ denotes the standard differential
\[
d_{\mathscr{B}B}(b_1\ox\dots\ox b_l)=\sum_i (-1)^ib_1\ox \dots\ox b_ib_{i+1}\ox\dots \ox b_l.
\]
on the bar {\it construction} 
\[
\mathscr{B}B=\cdots\to B^{\ox 3}\to B^{\ox 2}\to B\overset{0}\to k\to 0.
\]
The bar construction and bar resolution are related by the formula $\mathscr{B}ar^cB=B\ox\mathscr{B}B\ox B$.  To be precise, one arrives at the bar resolution by taking a ``twisted tensor product" between $B$ and the bar construction so that
\[
d_{\mathscr{B}ar^cB}=id_B\ox d_{\mathscr{B}B}\ox id_B+\text{extremal termal}.
\]

\begin{proposition}\label{prop:129}
Let $d_c$ denote the differential on $\bigoplus_n \uHom(B^{\ox n},B)$ induced by the restriction isomorphism (\ref{eq:196}).
\begin{enumerate}
\item $d_c$ is given in degree $0$ by the formula
\[
-d_c(f)(b)=(r^j\cdot b)\big(r_j\cdot f\big)(1)-f(1)b
\]
and in higher degree by the formula
\begin{equation}\label{eq:dR}
\ba{l}
(-1)^{|f|+1}d_c(f)(b\ox y\ox b')\\
=(r^j\cdot b)\big(r_j\cdot f\big)(y\ox b')+f(d_{\mathscr{B}B}(b\ox y\ox b'))+(-1)^{n+1}f(b\ox y)b',
\ea
\end{equation}
where $f$ is a homogenous map, $y\in B^{\ox n-2}$, and $b,b'\in B$.
\item Each differential $d_c:\uHom(B^{\ox n},B)\to \uHom(B^{\ox n+1},B)$ is a map in $\mathscr{Z}$.
\end{enumerate}
\end{proposition}

\begin{proof}
(1) Let $F$ be the preimage of $f$ along the restriction map.  Suppose $|f|=n>0$.  According to the previous lemma we have
\[
\ba{l}
(-1)^{|f|+1}d_c(f)(b\ox y\ox b')=d(F)(1\ox b\ox y\ox b'\ox 1)\\
=F(b\ox y\ox b'\ox 1)+(-1)^{n+1}F(1\ox b\ox y\ox b')+F(1\ox d_{\mathscr{B}B}(b\ox y\ox b')\ox 1)\\
=(r^j\cdot b)\big(r_j\cdot f\big)(y\ox b')+f(d_{\mathscr{B}B}(b\ox y\ox b'))+(-1)^{n+1}f(b\ox y)b'.
\ea
\]
The formula in degree $0$ is verified similarly.
\par

(2) We note first that, for $f\in \uHom(B^{\ox n},B)$, $b\in B$, $x\in B^{\ox n}$, the element
\[
(r^j\cdot b)\big(r_j\cdot f\big)(x)
\]
is the image of the monomial $f\ox b\ox x$ under the sequence of maps
\[
\uHom(B^{\ox n},B)\ox B\ox B^{\ox n}\overset{c\ox id}\longrightarrow B\ox\uHom(B^{\ox n},B)\ox B^{\ox n}\overset{id\ox\mathrm{pair}}\longrightarrow B\ox B \overset{\mathrm{mult}}\longrightarrow  B,
\]
all of which are morphisms in $\mathscr{Z}=H\mathrm{mod}$.  So for any $h\in H$ we have
\[
h\cdot \big((r^j\cdot b)(r_j\cdot f)(x)\big)=(r^jh_2\cdot b)(r_jh_1\cdot f)(h_3\cdot x).
\]
Replacing $h\ox b\ox x$ with $h_1\ox (S(h_3)\cdot b)\ox (S(h_2)\cdot x)$ then gives
\begin{equation}\label{eq:the1}
\ba{l}
h_1\cdot \big((r^jS(h_3)\cdot b)(r_j\cdot f)(S(h_2)\cdot x)\big)\\
=(r^jh_2S(h_5)\cdot b)(r_jh_1\cdot f)(h_3S(h_4)\cdot x)\\
=(r^jh_2\epsilon(h_3)S(h_4)\cdot b)(r_jh_1\cdot f)(\cdot x)\\
=(r^j\cdot b)(r_jh\cdot f)(x).
\ea
\end{equation}
One can also check the more obvious relations
\begin{equation}\label{eq:the2}
h_1\cdot \Big(f\big(d_{\mathscr{B}B}(S(h_2)\cdot x)\big)\Big) =h_1\cdot\Big(f\big(S(h_2)\cdot d_{\mathscr{B}B}(x))\Big)=(h\cdot f)\big(d_{\mathscr{B}B}(x)\big)
\end{equation}
and
\begin{equation}\label{eq:the3}
h_1\cdot \Big(f(S(h_3)\cdot x)(S(h_2)\cdot b')\Big)=\big(h_1\cdot f(S(h_2)\cdot x)\big)b'=(h\cdot f)(x)b'.
\end{equation}
So we see that for $x=b\ox y\ox b'$ we have
\[
\ba{l}
(-1)^{|f|+1}\big(h\cdot d_c(f)-d_c(h\cdot f)\big)(x)\\
=(-1)^{|f|+1} \Big(\big(h\cdot d_c(f)\big)(b\ox y\ox b')-d_c\big(h\cdot f\big)(b\ox y\ox b')\Big)\\
=h_1\cdot \Big((r^jS(h_3)\cdot b)(r_j\cdot f)\big(S(h_2)\cdot (y\ox b')\big)\Big)-(r^j\cdot b)\big(r_lh\cdot f\big)(y\ox b')\\
+h_1\cdot \Big(f\big(d_{\mathscr{B}B}(S(h_2)x)\big)\Big)-(h\cdot f)\big(d_{\mathscr{B}B}(x)\big)\\
+h_1\cdot \Big(f(S(h_3)\cdot (b\ox y))\big(S(h_2)\cdot b'\big)\Big)-(h\cdot f)(b\ox y)b'.
\ea
\]
Checking the above expressions (\ref{eq:the1})-(\ref{eq:the3}), we see that the final sum is $0$.  So $h\cdot d_c(f)=d_c(h\cdot f)$ and therefore $d_c$ is $H$-linear .
\end{proof}

\begin{definition}\label{defn:braided_cmplx}
We let $C^\bullet_c(B)$ denote the complex
\[
C^\bullet_c(B)=0\to \uHom(k,B)\overset{d_c^0}\to \uHom(B,B)\overset{d_c^1}\to \uHom(B^{\ox 2},B)\to\cdots,
\]
where $d_c$ is given by the formula (\ref{eq:dR}).  We call this complex the braided Hochschild cochain complex.
\end{definition}

Proposition \ref{prop:129} tells us that $C_c^\bullet(B)$ is a chain complex in $\mathscr{Z}$.  So we get

\begin{theorem}
\begin{enumerate}
\item The braided Hochschild cohomology $H_c^\bullet(B)$ is the cohomology of the braided Hochschild complex $C^\bullet_c(B)$.
\item The cohomology $H_c^\bullet(B)$ is an object in $\mathrm{dg}\mathscr{Z}$ (with vanishing differential).
\end{enumerate}
\end{theorem}

\begin{proof}
Statement (1) follows from Lemma \ref{lem:105} and the construction of the differential $d_c$.  Statment (2) follows from (1) and Proposition \ref{prop:129}.
\end{proof}

\begin{remark}
The complex $C^\bullet_c(B)$ may be legitimate even when we replace $\mathscr{Z}$ with an arbitrary braided multi-tensor category, with possibly nontrivial associator.  This is despite the fact that Definition \ref{def:ogdef} may make no sense in this setting.
\end{remark} 

\subsection{Relative braided cohomology}

Given any subalgebra $E\subset B$ in $\mathscr{Z}$ we can define the relative bar resolution $\mathscr{B}ar_E^c B$ to be the complex
\[
\cdots\to B\ox_E B^{\ox_E 2}\ox_E B\to B\ox_E B\ox_E B\to B\ox_E B\to 0 
\]
with the usual differential and right $B^{\unl{e}}$-action given by the same formula (\ref{eq:actn}).  We define the $E$-relative Hochschild cohomology as the cohomology
\[
H^\bullet_{c,E}(B)=H^\bullet\big(\Hom_{\mathrm{mod}B^{\unl{e}}}(\mathscr{B}ar_E^c B,B)\big).
\]
We similarly let $\mathscr{B}_EB$ denote the relative bar construction
\[
\mathscr{B}_EB=\cdots\to B^{\ox_E 3}\to B\ox_E B\to B\overset{0}\to E\to 0
\]
with the usual differential, as above.
\par

We have an embedding
\begin{equation}\label{eq:468}
\Hom_{\mathrm{mod}B^{\unl{e}}}(\mathscr{B}ar^c_EB, B)\to \Hom_{\mathrm{mod}B^{\unl{e}}}(\mathscr{B}ar^cB, B)=C^\bullet_c(B)
\end{equation}
of cochain complexes dual to the projection $\mathscr{B}ar^cB\to \mathscr{B}ar^c_EB$.  In order to think clearly about the image of this embedding we introduce a new category.

\begin{definition}
Let $E$ be an algebra in $\mathscr{Z}$.  We let $_E\mathscr{Z}_E$ denote the category of objects $M$ in $\mathscr{Z}$ with an additional $E$-bimodule structure such that the action $E\ox M\ox E\to M$ is a map in $\mathscr{Z}$.  Morphisms in this category will be maps in $\mathscr{Z}$ which are also $E$-bimodule maps.
\end{definition}

Both $B^{\ox n}$ and $B^{\ox_E n}$ will be objects in $_E\mathscr{Z}_E$, for example.  The projections $B^{\ox n}\to B^{\ox_E n}$ will be maps in $_E\mathscr{Z}_E$.
\par 

The category $_E\mathscr{Z}_E$ is actually the category of modules over a certain smash-product-esque algebra~\cite[Lemma 3.2]{kaygun}.  Whence $_E\mathscr{Z}_E$ will be abelian with enough projectives.  There is a natural functor
\begin{equation}\label{314}
Forget:{_E\mathscr{Z}_E}\to \mathrm{mod}E^{\unl{e}}
\end{equation}
which sends an object $M$ to the vector space $M$ with the right $E^{\unl{e}}$-action $x\cdot(w\ox w')=(r^j\cdot w)(r_j\cdot x)w'$ for each $x\in M$, $w,w'\in E$.
\par

Given any $M$ and $N$ in $_E\mathscr{Z}_E$ the maps in $\Hom_{\mathrm{mod}E^{\unl{e}}}(M,N)$ are exactly those $k$-linear maps $f$ which satisfy the additional relations
\begin{equation}\label{eq:1}
f(xw)=f(x)w
\end{equation}
and
\begin{equation}\label{eq:2}
f(wx)=(r^j\cdot w)\big(r_j\cdot f)(x)
\end{equation}
for each $x\in M$, $w\in E$.  The relation (\ref{eq:1}) is obviously just right $E$-linearity while (\ref{eq:2}) should be familiar from Lemma \ref{lem:105}.
\par

Recall that the set of inner homs $\uHom(M,N)$ in $\mathscr{Z}$ is equal to the set of $k$-linear maps, after we forget the $\mathscr{Z}$-structure.

\begin{lemma}\label{lem:w/e}
In each degree $n$, the image of each embedding (\ref{eq:468}) is exactly the subset
\[
\Hom_{\mathrm{mod}E^{\unl{e}}}(B^{\ox_E n},B)\subset \uHom(B^{\ox n},B).
\]
\end{lemma}
It may actually be more informative to think instead about the sequence of inclusions
\[
\Hom_{\mathrm{mod}E^{\unl{e}}}(B^{\ox_E n},B)\subset \uHom(B^{\ox_E n},B)\subset \uHom(B^{\ox n},B)
\]
\begin{proof}
This follows from the formula for the identification (\ref{eq:og_rels}) of Lemma \ref{lem:105}, as well as the description of the maps in $\Hom_{\mathrm{mod}E^{\unl{e}}}$ given in equation (\ref{eq:1}) and (\ref{eq:2}).
\end{proof}

\begin{lemma}\label{lem:309}
For any $M,N$ in $_E\mathscr{Z}_E$, the embedding $\Hom_{\mathrm{mod}E^{\underline{e}}}(M,N)\to\uHom(M,N)$ realizes $\Hom_{\mathrm{mod}E^{\unl{e}}}(M,N)$ as an object in $\mathscr{Z}$. 
\end{lemma}

\begin{proof}
We need to show that each $\Hom_{\mathrm{mod}E^{\underline{e}}}(M,N)$ is stable under the action of $H$, where $\mathscr{Z}=H\mathrm{mod}$.  A straightforward computation shows that the subspace of right $E$-linear maps in a $H$-submodule in $\uHom(M,N)$, so we need only concern ourselves with the relation (\ref{eq:2}).  
\par

Suppose $f\in \uHom(M,N)$ satisfies the relation (\ref{eq:2}) and take any $h\in H$.  Then
\[
(h\cdot f)(wx)=h_1\cdot \Big((r^jS(h_3)\cdot w)\big(r_j\cdot f\big)(S(h_2)\cdot x)\big)\Big).
\]
We find just as in the proof of Proposition \ref{prop:129} (2) the equality
$$
h_1\cdot \Big((r^jS(h_3)\cdot w)\big(r_j\cdot f\big)(S(h_2)\cdot x)\big)\Big)=(r^j\cdot w)\big(r_jh\cdot f\big)(x),
$$
which gives the desired relation
\[
(h\cdot f)(wx)=(r^j\cdot w)\big(r_jh\cdot f\big)(x).
\]
So we see that $\Hom_{\mathrm{mod}E^{\underline{e}}}(M,N)\subset \uHom(M,N)$ is an $H$-submodule.
\end{proof}

In order to emphasize the natural $\mathscr{Z}$-structure on these hom sets we adopt

\begin{notation}
For any $M$ and $N$ in $_E\mathscr{Z}_E$ we let $\uHom_E(M,N)$ denote the set of maps $\Hom_{\mathrm{mod}E^{\unl{e}}}(M,N)$ along with its $\mathscr{Z}$-structure induced by the inclusion $\Hom_{\mathrm{mod}E^{\underline{e}}}(M,N)\to\uHom(M,N)$.
\end{notation}

We remark that this hom set may not truly be an inner hom set in a category, despite the fact that we have elected to adopt the underline notation.
\par

Lemma \ref{lem:309} implies that the image of $\Hom_{\mathrm{mod}B^{\unl{e}}}(\mathscr{B}ar^c_EB,B)$ in $C^\bullet_c(B)$ will be a subcomplex in $\mathscr{Z}$.  So the complex $\Hom_{\mathrm{mod}B^{\unl{e}}}(\mathscr{B}ar^c_EB,B)$ becomes naturally an object in $\mathrm{dg}\mathscr{Z}$, and the relative braided Hochschild cohomology will again be $\mathscr{Z}$-valued.
\par

We now write down an explicit complex $C^\bullet_{c,E}(B)$ analogous to the braided Hochschild cochain complex constructed in the non-relative instance: first, Lemma \ref{lem:w/e} tells us that the image of the embedding (\ref{eq:468}) is the graded subspace
\[
\bigoplus_{n\in \mathbb{Z}}\uHom_E(B^{\ox_E n},B)\subset C^\bullet_c(B).
\]
Second, since (\ref{eq:468}) is a chain map we see that this subspace is a subcomplex under the differential on $C^\bullet_c(B)$.  So restricting to the generators now provides a cochain complex isomorphism
\[
\Hom_{\mathrm{mod}B^{\unl{e}}}(\mathscr{B}ar^c_E B,B)\overset{\cong}\to \big(\bigoplus_{n\in \mathbb{Z}}\uHom_E(B^{\ox_E n},B), d_c\big),
\]
where the differential $d_c$ is given by the formula (\ref{eq:dR}).

\begin{definition}
We let $C^\bullet_{c,E}(B)$ denote the complex
\[
C^\bullet_{c,E}(B)=0\to \uHom_E(E,B)\overset{d_c^0}\to \uHom_E(B,B)\overset{d_c^1}\to \uHom_E(B^{\ox_E 2},B)\overset{d_c^2}\to\cdots.
\]
\end{definition}

The following proposition is now clear.

\begin{proposition}
We have $H_{c,E}^\bullet(B)=H^\bullet\big(C^\bullet_{c,E}(B)\big)$, and the cohomology $H^\bullet_{c,E}(B)$ is an object in $\mathrm{dg}\mathscr{Z}$.
\end{proposition}

Our main intent in introducing the relative cohomology is to provide a relationship between the braided Hochschild cohomology and standard Hochschild cohomology.  With the above information we will be able to relate the standard Hochschild cohomology to the {\it relative} braided cohomology.  (This is done in Section \ref{sect:id_w_hh}.)  We would like to establish finally some conditions under which the relative and non-relative braided Hochschild cohomologies agree.

\begin{definition}
We say $E$ is separable in $\mathscr{Z}$ if the category $_E\mathscr{Z}_E$ is semisimple.
\end{definition}

\begin{lemma}\label{lem:semsim}
\begin{enumerate}
\item For any $M$ and $N$ in $_E\mathscr{Z}_E$ there is an equality
\[
\Hom_{_E\mathscr{Z}_E}(M,N)=\uHom_E(M,N)^{\mathscr{Z}}.
\]
\item If $\mathscr{Z}$ is semisimple, then $E$ is separable in $\mathscr{Z}$ whenever $E^{\underline{e}}$ is semisimple as an ordinary $k$-algebra.
\end{enumerate}
\end{lemma}

Note that both $\Hom_{_E\mathscr{Z}_E}(M,N)$ and $\uHom_E(M,N)$ are subsets in $\Hom_k(M,N)$.  So it makes sense to propose that the two sets are equal.

\begin{proof}
(1) Maps in both hom sets are right $E$-linear, as well as maps in $\mathscr{Z}$.  So we need only concern ourselves with left $E$-linearity.  Any $H$-invariant (i.e. $H$-linear) function $f$ in $\uHom_E(M,N)$ will satisfy $h\cdot f=\epsilon(h)f$ for each $h\in H$.  Therefore, for any $w\in E$, $m\in M$, we have
\[
f(wm)=(r^j\cdot w)\big(r_j \cdot f\big)(m)=(r^j\epsilon(r_j)\cdot w)f(m)=wf(m).
\]
Whence $f$ is left $E$-linear.  So $\uHom_E(M,N)^\mathscr{Z}\subset \Hom_{_E\mathscr{Z}_E}(M,N)$.  
\par

Recall now that maps in $_E\mathscr{Z}_E$ are $H$-linear, and hence $H$-invariant.  Therefore, for any $g\in \Hom_{_E\mathscr{Z}_E}(M,N)$,
\[
g(wm)=wg(m)=(r^j\cdot w)\big(r_j\cdot g\big)(m).
\]
So we get the opposite inclusion as well, and hence an equality.
\par

(2) Note that $\mathscr{Z}$ is semisimple if and only if the invariants functor $(-)^\mathscr{Z}$ is exact.  So, if $\mathscr{Z}$ and $E^{\unl{e}}$ are semisimple then for any $M$ in $_E\mathscr{Z}_E$ the hom functor $\Hom_{_E\mathscr{Z}_E}(M,-)$ is exact, since it will be the composite of the exact functors 
\[
(-)^\mathscr{Z}\circ \uHom_E(M,-)=(-)^\mathscr{Z}\circ \Hom_{\mathrm{mod}E^{\unl{e}}}(M,-)
\]
by (1).  So all objects in $_E\mathscr{Z}_E$ are projective, and $E$ is separable in $\mathscr{Z}$.
\end{proof}

The following is a braided version of a result of Gerstenhaber and Schack~\cite{gerstenhaberschack86}.

\begin{proposition}\label{prop:relVreg}
If $E$ is separable in $\mathscr{Z}$ then
\begin{enumerate}
\item the inclusion $C_{c,E}^\bullet(B)\to C_c^\bullet(B)$ is a quasi-isomorphism,
\item there induced map $H^\bullet_{c,E}(B)\to H^\bullet_c(B)$ is an isomorphism in $\mathrm{dg} \mathscr{Z}$.
\end{enumerate}
\end{proposition}

\begin{proof}
The second statement follows from the first.  So we address (1).  First, note that for any $M$ in $_E\mathscr{Z}_E$, $B\ox M\ox B$ and $B\ox_E M\ox_E B$ become objects in mod$B^{\unl{e}}$ under the obvious action(s) 
\[
(a\ox m\ox a')\ast (b\ox b')=(r^j\cdot b)\big(r_j\cdot (a\ox m\ox a')\big)b'.
\]
As was the case above, the braiding gives an isomorphism between each $B\ox M\ox B$ and the free $B^{\unl{e}}$-module $M\ox B^{\unl{e}}$.
\par

Now, If $E$ is separable in $\mathscr{Z}$ then for any $M$ in $_E\mathscr{Z}_E$ the action map $m_M:E\ox M\ox E\to M$ will admit a splitting $\sigma_M:M\to E\ox M\ox E$ in $_E\mathscr{Z}_E$.  The multiplication map $m:E\ox E\to E$ will also be split in $_E\mathscr{Z}_E$ by some $\sigma_E:E\to E\ox E$.  Whence the projections
\[
B\ox M\ox B=B\ox_E E\ox M\ox E\ox_E B\overset{id\ox m_M\ox id}\to B\ox_E M\ox_E B
\]
and
\[
B\ox B=B\ox_E E\ox E\ox_E B\overset{id\ox m\ox id}\to B\ox_E B
\]
will be split in mod$B^{\unl{e}}$ by $1\ox \sigma_M\ox 1$ and $1\ox \sigma_E\ox 1$ respectively.  So $B\ox_E M\ox_E B$ and $B\ox_E B$ are summands of free modules and hence projective in mod$B^{\unl{e}}$.  
\par

By the above general argument we find that each $B\ox_E B^{\ox_E n}\ox_E B$ is projective and $\mathscr{B}ar^c_EB\to B$ is a projective resolution of $B$ in mod$B^{\unl{e}}$.  Consequently the projection $\mathscr{B}ar^cB\to \mathscr{B}ar^c_EB$ is a homotopy equivalence, and applying the functor $\Hom_{\mathrm{mod}B^{\unl{e}}}(-,B)$ produces a quasi-isomorphism
\begin{equation}\label{eq:480}
\Hom_{\mathrm{mod}B^{\unl{e}}}(\mathscr{B}ar_E^cB,B)\overset{\sim}\to \Hom_{\mathrm{mod}B^{\unl{e}}}(\mathscr{B}ar^cB,B).
\end{equation}
Under the identifications
\[
\Hom_{\mathrm{mod}B^{\unl{e}}}(\mathscr{B}ar_E^cB,B)=C^\bullet_{c,E}(B)\ \ \mathrm{and}\ \ \Hom_{\mathrm{mod}B^{\unl{e}}}(\mathscr{B}ar^cB,B)=C^\bullet_{c,E}(B),
\]
which are both given by restriction, (\ref{eq:480}) becomes the canonical embedding $C_{c,E}^\bullet(B)\to C_c^\bullet(B)$.  We see now that the embedding is a quasi-isomorphism.
\end{proof}

\section{Identifications with standard Hochschild cohomology for smash products}
\label{sect:id_w_hh}

We fix $E$ to be a finite dimensional Hopf algebra, $A$ an $E$-module algebra, $\mathscr{Z}=YD^E_E$, and $B=A\ast E$.  We take $A\ast E$ to be an algebra in $YD^E_E$ as described in Section \ref{sect:smashprud}.  We can recover $A$ in a functorial manner from $B$ as the coinvariants $A=B^{coE}$.  The subalgebra $E\subset B$ becomes a subalgebra in $YD^E_E$ under the adjoint action of $E$ on itself and regular right coaction.  So we may consider the relative cohomology $H_{c,E}^\bullet(B)$.  
\par

Recall from the relation (\ref{eq:1}) that any map in $\uHom_E(B^{\ox_E n},B)$ will be right $E$-linear.

\begin{lemma}
When $B=A\ast E$, then the inclusion
\[
\uHom_E(B^{\ox_E n}, B)\subset \Hom_{\mathrm{mod}E}(B^{\ox_E n},B)
\]
is an equality.
\end{lemma}

\begin{proof}
We let juxtaposition denote the multiplication in $B$, and the dot $\cdot$ denote the action of $E^{op}$ according to the Yetter-Drinfeld structure.  Take $D$ to be the double of $E$ so that $YD^E_E=D\mathrm{mod}$.  We verify that for any right $E$-linear $f:B^{\ox_E n}\to B$ the equation (\ref{eq:2}) will already be satisfied.
\par

Note that, since each tensor product is over $E$, we will have for any $w\in \G^{op}\subset D$ identities
\[
\ba{rl}
w\cdot (b_1\ox\dots\ox b_l)&=(S_E(w_1)b_1w_2)\ox (S_E(w_3) b_2 w_4)\ox\dots\ox (S_E(w_{2l-1})b_lw_{2l})\\
&=(S_E(w_1) b_1)\ox b_2\ox\dots\ox (b_l w_2)
\ea
\]
and hence, for any $w\in E^{op}\subset D$ and $x\in B^{\ox_E n}$,
\[
wx=(S^{-1}_E(w_2)\cdot x)w_1.
\]
It follows that for any right $E$-linear $f:B^{\ox_E n}\to B$ we have
\[
\ba{rll}
f(wx)&=f\big((S^{-1}_E(w_2)\cdot x)w_1\big)\\
&=f\big(S^{-1}_E(w_2)\cdot x\big)w_1 & \text{by right $E$-linearity}\\
&=w_1S_E(w_2)\Big(f\big(S^{-1}_E(w_4)\cdot x\big)w_3\Big)\\
&=w_1\Big(S_E(w_2)f\big(S^{-1}_E(w_4)\cdot x\big)w_3\Big)\\
&=w_1\Big(w_2\cdot \big(f\big(S^{-1}_E(w_3)\cdot x\big)\big)\Big)\\
&=w_1\big(w_2\cdot f\big)(x)
\ea
\]
Now if we view $E$ as a subcomodule in $B$ then $w_1\ox w_2=w_0\ox w_1=r^j\cdot w\ox r_j$ and the final equation above becomes
\[
f(wx)=(r^j\cdot w)\big(r_j\cdot f\big)(x).
\]
Whence (\ref{eq:2}) is satisfied.
\end{proof}

By the lemma, when $B=A\ast\G$, the complex $C^\bullet_{c,E}(B)$ now appears as
\[
\ba{l}
C^\bullet_{c,E}(B)\\
=0\to \Hom_{\mathrm{mod}E}(E,B)\overset{d_c^0}\to \Hom_{\mathrm{mod}E}(B,B)\overset{d_c^1}\to \Hom_{\mathrm{mod}E}(B^{\ox_E 2},B)\overset{d_c^2}\to\cdots.
\ea
\]
The complex $C^\bullet_{c,E}(B)$ in this case is exactly the ``intermediate complex" of~\cite[Ch. 3]{negronthesis}.\footnote{At the point of~\cite{negronthesis} the relation with braided cohomology was still obscured.}
\par

In the following proposition the notation $C^\bullet(A,B)$ denotes the standard, i.e. non-braided, Hochschild cochain complex of $A$ with coefficients in $B$.  Note that restricting along the embeddings $A^{\ox n}\to B^{\ox_E n}$ produces vector space maps
\[
C_{c,E}^n(B)=\uHom_{E}(B^{\ox_E n},B)\to \Hom_k(A^{\ox n},B)=C^n(A,B)
\]
in each degree $n$.

\begin{proposition}\label{prop:relVHH}
Restriction provides an isomorphism of chain complexes $C_{c,E}^\bullet(B)\overset{\cong}\to C^\bullet(A,B)$ and subsequent isomorphism of cohomologies $H_{c,E}^\bullet(B)\overset{\cong}\to HH^\bullet(A,B)$.
\end{proposition}

\begin{proof}
The fact that the restriction maps are isomorphisms follows immediately from the fact that the inclusion $A^{\ox n}\ox E=A^{\ox (n-1)}\ox B\to B^{\ox_E n}$ is an isomorphism of right $E$-modules.  As for the differential we note that, by coinvariance of elements in $A$ and the formula for the differential $d_c$ given at (\ref{eq:dR}), we have
\[
\ba{l}
(-1)^{n+1}d_c(f)(a_1\ox\dots\ox a_{n+1})\\
=a_1f(a_2\ox\dots\ox a_{n+1})+(-1)^{n+1}f(a_1\ox\dots a_{n})a_{n+1}\\
+\sum_{i=1}^{n} (-1)^if(a_1\ox\dots a_ia_{i+1}\ox\dots\ox a_{n+1})
\ea
\]
for any degree $n$ function $f\in C^n_c(B)$ and $a_i\in A\subset B$.  This is exactly the non-braided Hochschild differential.
\end{proof}

As a consequence of the above isomorphism, and the fact that $C^\bullet_{c,E}(B)$ is a complex in $YD^E_E$, we see that the Hochschild complex $C^\bullet(A,B)$ is a complex in $YD^E_E$.  We see also that the cohomology is a graded object in $YD^E_E$.  
\par

One can check easily that the right $E$-action implicit in this Yetter-Drinfeld structure is exactly the right action considered in~\cite{stefan,N2}.  The coaction is induced by the coaction on $B=A\ast E$.  Namely, by coinvariance of the elements in $A$, we will have for any $f\in C^\bullet(A,B)$, $\xi\in E^\ast$, and $x\in \mathscr{B}A$
\[
\big(\xi\cdot f\big)(x)=\xi_1\cdot \big(f(S(\xi_2)\cdot x)\big)=\xi\cdot \big(f(x)\big).
\]
Written another way, the coproduct $f_0\ox f_1$ of $f$ is the unique element so that
\[
f_0(x)\ox f_1=\big(f(x)\big)_0\ox\big(f(x)\big)_1
\]
for each $x\in \mathscr{B}A$.  In the case of a finite group ring $G$, for example, the $G$-coaction on $C^\bullet(A,B)=C^\bullet(A,A\ast G)$ corresponds to the obvious $G$-grading
\[
C^\bullet(A,A\ast G)=C^\bullet(A,\bigoplus_{g\in G}Ag)=\bigoplus_{g\in G} C^\bullet(A,Ag).
\]
\par

We now want to address the question of when $H^\bullet(A,B)$ is equal to the non-relative braided Hochschild cohomology.  In light of Proposition \ref{prop:relVreg}, it suffices to understand when $E$ is separable in $YD^E_E$.  We will see that this occurs when $E$ is both semisimple and cosemisimple.

\begin{lemma}
The vector space equality $E=E^{\unl{op}}$ is an equality of algebras in $YD^E_E$.
\end{lemma}

\begin{proof}
The equality is already an identification in $YD^E_E$, so we need only deal with the multiplication.  We denote the adjoint action of $E^{op}\subset D$ by a dot $\cdot$, and multiplication by juxtaposition.  For any $w,w'\in E$,
\[
w\cdot_{\unl{op}}w'=(w'_1)(w_2'\cdot w)=(w'_1)(S(w_2')ww_3')=(w'_1S(w_2'))ww_3'=ww',
\]
and we are done.
\end{proof}

\begin{lemma}\label{lem:510}
If $E$ is semisimple and cosemisimple then the algebra $E^{\unl{e}}$ is separable in $YD^E_E$.
\end{lemma}

\begin{proof}
Since $E$ is both semisimple and cosemisimple the double $D$ is semisimple (see e.g.~\cite{M}).  Hence $YD^E_E=D\mathrm{mod}$ is semisimple.  It therefore suffices to establish semisimplicity of the braided enveloping algebra $E^{\unl{e}}$, by Lemma \ref{lem:semsim}.  By the previous lemma we have $E^{\unl{e}}=E\unl{\ox}E$.  Let $E_1$ and $E_2$ denote the left and right copies of $E$ in the product $E\unl{\ox}E$.  Then, by checking the multiplication
\[
(a\ox b)(a'\ox b')=a(a'_1)\ox (a'_2\cdot b)b'=a(a'_1)\ox (S(a'_2)ba'_3)b'
\]
we see that $E_1\unl{\ox}E_2$ is the smash product $E_1\ast E_2$, where $E_2$ is a right $E_1$-module algebra under the adjoint action.  Now we have the standard identification of maps over the smash product as the invariants
\[
\Hom_{\mathrm{mod}E^{\unl{e}}}(-,-)=\Hom_{\mathrm{mod}E_1\ast E_2}(-,-)=\Hom_{\mathrm{mod}E_2}(-,-)^{E_1}.
\]
Semisimplicity of $E=E_1=E_2$ then implies exactness of the functor
\[
\Hom_{\mathrm{mod}E^{\unl{e}}}(M,-)=\Hom_{\mathrm{mod}E_2}(M,-)^{E_1}
\]
for arbitrary $M$ in $\mathrm{mod}E^{\unl{e}}$.  So all objects in $\mathrm{mod}E^{\unl{e}}$ are projective, and $E^{\unl{e}}$ is semisimple.
\end{proof}

Lemma \ref{lem:510} in conjunction with Proposition \ref{prop:relVreg} then gives

\begin{theorem}\label{thm:relVreg}
Suppose that $E$ is semisimple and cosemisimple, and $B=A\ast E$.  Then there is a canonical isomorphism between the $E$-relative braided cohomology and non-relative braided cohomology $H_{c,E}^\bullet(B)\cong H_c^\bullet(B)$ in $\mathrm{dg} YD^E_E$.
\end{theorem}

The following Corollary was proved in the unpublished work of Schedler and Witherspoon in the case of a finite group $E=kG$ in characteristic $0$~\cite{schedlerwitherspoon}.  The methods employed in~\cite{schedlerwitherspoon} were quite different.

\begin{corollary}(Generalized Schedler-Witherspoon Isomorphism)\label{cor:HcVHH}
When $E$ is semisimple and cosemisimple, and $B=A\ast E$, there is a canonical isomorphism $HH^\bullet(A,B)\cong H^\bullet_c(B)$.
\end{corollary}

\begin{proof}
By Proposition \ref{prop:relVHH} and Theorem \ref{thm:relVreg} we have $HH^\bullet(A,B)\cong H_{c,E}^\bullet(B)\cong H_c^\bullet(B)$.
\end{proof}

As noted above, the Yetter-Drinfeld structure on $HH^\bullet(A,B)$ induced by this isomorphism is the expected one, with the action from~\cite{stefan,N2} and coaction induced from $B=A\ast E$.  As a consequence of the above results and \c{S}tefan's spectral sequence we also have

\begin{corollary}
When $E$ is semisimple and cosemisimple there is a vector space identification $HH^\bullet(B)=H_c^\bullet(B)^E$ and for general $E$ there is a spectral sequence
\[
\Ext_{\mathrm{mod}E}\big(k,H_{c,E}^\bullet(B)\big)\Rightarrow HH^\bullet(B)
\]
\end{corollary}

\begin{proof}
\c{S}tefan's spectral sequence~\cite{stefan} appears as
\[
\Ext_{\mathrm{mod}E}\big(k,HH^\bullet(A,B)\big)\Rightarrow HH^\bullet(B).
\]
So the result follows from Proposition \ref{prop:relVHH} and Theorem \ref{thm:relVreg}.
\end{proof}

\subsection{An application to deformation theory}\label{sect:defthry}

Take $B=A\ast E$.  For this subsection we would like to employ the normalized complex $\bar{C}^\bullet_{c,E}(B)$.  This consists of all maps in the complex $C^\bullet_{c,E}(B)$ which vanish on any monomial $b_1\ox\dots\ox b_n$ with $b_i=1$ for some $i$.  The inclusion $\bar{C}^\bullet_{c,E}(B)\to C^\bullet_{c,E}(B)$ will be a quasi-isomorphism, for the same reason as it is the case in the non-braided setting~\cite{Gi,gerstenhaberschack86}.  Namely, the kernel of the projection $\mathscr{B}ar^c_EB\to \bar{\mathscr{B}ar^c_E}B$, to which the embedding $\bar{C}^\bullet_{c,E}(B)\to C^\bullet_{c,E}(B)$ is dual, will (still) be contractible.
\par

In the following we take $L(B)$ to be the shifted, normalized, non-braided Hochschild complex $\Sigma \bar{C}^\bullet(B)$, with its usual dg Lie algebra structure, and we let $L(B;E)$ denote the $E$-invariants of the shifted, normalized, braided Hochschild complex $\Sigma \bar{C}^\bullet_c(B)^E$.  One can see easily that the differential on invariant (i.e. left $E$-linear) functions in $\bar{C}^\bullet_c(B)$ reduces to the usual, non-braided, Hochschild differential and that the inclusion
\[
\bar{C}^\bullet_c(B)^E\to \bar{C}^\bullet(B)
\]
induces a dg Lie structure on the shifted complex $L(B;E)$.  The inclusion identifies the invariant functions in $\bar{C}^\bullet_c(B)$ with left $E$-linear functions in $\bar{C}^\bullet(B)$.

\begin{theorem}\label{thm:defthry}
Suppose $E$ is both semisimple and cosemisimple and take $B=A\ast E$.  The inclusion $L(B;E)\to L(B)$ is a quasi-isomorphism of dg Lie algebras.
\end{theorem}

Without getting completely sidetracked by the issue, the above corollary says that when $E$ is semisimple and cosemisimple the formal deformation theory of $B$ as an associative algebra, and $B$ as an $E$-module algebra, will be equivalent.  See~\cite{manetti05,KSdeformation,negronthesis,yau08,fregier09,lurie11,dolgushev15}.
\par

We note that our complex $L(B;E)$ representes the formal deformation theory of $B$ with {\it fixed} $E$-module structure.  However, since $E$ is semisimple $B$ will not be deformable as an $E$-module, so the distinction should be immaterial (see also~\cite{yau08}).

\begin{proof}[Proof of Theorem \ref{thm:defthry}]
In this proof we take $C^\bullet_\ast$ to mean the normalized complex, to ease notation.  We have the quasi-isomorphism of complexes $C^\bullet(A,B)=C^\bullet_{c,E}(B)\to C^\bullet_c(B)$ and taking invariants gives another quasi-isomorphism $C^\bullet(A,B)^E\to C^\bullet_c(B)^E$ (since $E$ is semisimple and the invariants functor is therefore exact).  So, to see that the shift of the map $C^\bullet_c(B)^E\to C^\bullet(B)$ is a quasi-isomorphism it suffices to show that the composite map $C^\bullet(A,B)^E\to C^\bullet_c(B)^E\to C^\bullet(B)$ is a quasi-isomorphism.
\par

But now $C^\bullet(A,B)^E$ is canonically identified with the relative Hochschild complex $C^\bullet_E(B)$.  Since $E$ is separable it then follows from~\cite{gerstenhaberschack86} that the inclusion
\[
\Sigma C^\bullet(A,B)^E=\Sigma C^\bullet_E(B)\to \Sigma C^\bullet(B)
\]
is in fact a quasi-isomorphism.  
\end{proof}

\subsection{Cohomology for crossed products and $J$-twists}
\label{sect:w_cocycle}

Take $E$ a finite dimensional Hopf algebra and $B=A\ast E$.  Let $J\in E^\ast\ox E^\ast$ be an invertible dual cocycle (or, {\it twist})~\cite{montgomery04}.  We can consider this element also as an element in the second tensor power of the double $D\ox D$.  For any such $J$, and (dg) algebra $\Omega$ in $YD^E_E=D\mathrm{mod}$ we can form the (dg) algebra $\Omega_J$ by altering the multiplication
\[
\omega\cdot_J\omega'=(J_l\cdot \omega)(J^l\cdot \omega'),
\]
where $J=J_l\ox J^l$.  (There is an implicit sum here $J=\sum_l J_l\ox J^l$.)  This new object will be a (dg) algebra in the category of modules over the twists $D^J$~\cite{montgomery04}, and the Hopf algebra $D^J$ will be quasitriangular with new $R$-matrix $R=J_{21}^{-1}RJ$~\cite{EGNO}.  So, the algebra $\Omega_J$ is an algebra in the braided category $\mathrm{dg}D^J\mathrm{mod}$.
\par

In the case that $E$ is cocommutative (e.g. a group algebra) we arrive at all crossed products $A\ast_\alpha E$, with $\alpha$ taking values in $k$, as $(A\ast E)_J$ where $J$ is the element identified with $\alpha$ under the isomorphism $E^\ast\ox E^\ast\cong (E\ox E)^\ast$.  So we see that many crossed products of interest live naturally in a braided monoidal categories as well.  For example, we can consider the motivating examples of~\cite{caldararuetal}.  I make the following claim.

\begin{claim}
Take $B=(A\ast E)$ and $J$ an invertible dual cocycle in $E^\ast\ox E^\ast\subset D\ox D$.  Then there is a natural identification
\[
HH^\bullet(A,B_J)=H^\bullet_{c,E_J}(B_J)
\]
and when $E$ is both semisimple and cosemisimple we get a natural isomorphism with the non-relative braided cohomology
\[
HH^\bullet(A,B_J)\cong H^\bullet_c(B_J).
\]
Whence the Hochschild cohomology has the structure of an object in $\mathrm{dg}D^J\mathrm{mod}$.
\end{claim}

With the claim one can provide an analog of Theorem \ref{thm:defthry} for certain crossed products.  At this point the proof of the claim is quite arduous (although it can be presented clearly).  So we forgo this path and opt instead for a happy in between with the following proposition.

\begin{proposition}\label{prop:w_cocycle}
Take $B=A\ast E$ and any $J$ as above.
\begin{enumerate}
\item There is an equality of $HH^\bullet(A)$-bimodules $HH^\bullet(A,B_J)=HH^\bullet(A,B)$, which also holds at the cochain level.
\item There is an equality of dg algebras
\[
C^\bullet(A,B)_J=C^\bullet(A,B_J)
\]
and also an equality of Hochschild cohomology rings
\[
HH^\bullet(A,B)_J=HH^\bullet(A,B_J).
\]
\item The equality of (2) gives the Hochschild cohomology of $B_J$ the natural structure of an object in the category $\mathrm{dg}D^J\mathrm{mod}$.
\end{enumerate}
\end{proposition}

We adopt some of the notation (but none of the results) of Section \ref{sect:product} for the proof.

\begin{proof}
We prove (1) last, and begin with (2).  Note that $A^{\ox n}\subset B^{\ox_E n}$ is exactly the coinvariants.  So $J=J_l\ox J^l$ will act trivially on the tensor powers of $A$.  This fact also implies that for any $a\in A$, and $b\in B$, we will have $a\cdot_J b=ab$ and $b\cdot_J a=ba$, so that $B_J=B$ as an $A$-bimodule.  Thus we have an equality of cochain complexes $C^\bullet(A,B)=C^\bullet(A,B_J)$.
\par

Take now any functions $f,g\in C^\bullet(A,B)=C^\bullet(A,B_J)$, and any $x\in A^{\ox n}$.  Then, when we take the product in $C^\bullet(A,B)_J$ we get
\[
\ba{rl}
\pm f\cup_J g(x)&=\big(J_l\cdot f\big)(x_1)\big(J^l\cdot g)(x_2)\\
&=\big({J_l}_1\cdot f(S({J_l}_2)\cdot x_1)\big)\big(J^l_1\cdot g(S(J^l_2)\cdot x_2)\big)\\
&=\big(J_l\cdot f(x_1)\big)\big(J^l\cdot g(x_2)\big)\ \ \text{since $x_i$ is coinvariant}\\
&=f(x_1)\cdot_J g(x_2).
\ea
\]
This last expression is equal to the cup product in $C^\bullet(A,B_J)$, and we are done.  Take cohomology to get the identification for $HH^\bullet$.  
\par

For (3) we note that $(-)_J$ restricts to a functor from $D$-module algebras to $D^J$-module algebras.  For (1) we note that $HH^\bullet(A)$ sits inside of the coinvariants in $HH^\bullet(A,B)$.  The same statement holds at the cochain level.  So whenever $f$ or $g$ is in $C^\bullet(A)\subset C^\bullet(A,B)=C^\bullet(A,B_J)$ we have
\[
f(x_1)\cdot_J g(x_2)=\big(J_l\cdot f(x_1)\big)\big(J^l\cdot g(x_2)\big)=f(x_1)g(x_2),
\]
since $(\epsilon\ox id)(J)=(id\ox\epsilon)(J)=1$.
\end{proof}

The $D^J$-module structure arrived at via twisting is the same $D^J$-module structure we arrive at by the more robust means suggested in the claim.  Namely, if we forget the coproduct on $D^J$ and the product on the Hochschild cohomology, we have $HH^\bullet(A,B_J)=HH^\bullet(A,B)$ as a $D=D^J$-module.  As a consequence we find that the $E$-action on Hochschild cohomology arrived at via the algebra (but not Hopf) embedding $E^{op}\to D^J$ does {\it not} agree with those considered in~\cite{stefan,guichardet01,GuGu}.  So it seems that we need to leave our setting, and consider~\cite{stefan} independently, in order to produce a spectral sequence relating the cohomology $H_c^\bullet(B_J)$ to the Hochschild cohomology $HH^\bullet(B_J)$.  At the moment this point of confusion remains unresolved.

\section{The cup product and braided commutativity of braided Hochschild cohomology}
\label{sect:product}

In this section we show that the complex $C^\bullet_c(B)$ and cohomology $H_c^\bullet(B)$ both admit canonical products, and that the cohomology is a braided commutative algebra under this product.  The same will hold for the relative cohomology as well.  Applications to smash products are given in Section \ref{sect:consequences}, as well as Corollary \ref{cor:in_center}.

\subsection{The braided category $\mathrm{dg} \mathscr{Z}$}

For any braided tensor category $\mathscr{Z}=H\mathrm{mod}$ the category $\mathrm{dg} \mathscr{Z}$ of complexes over $\mathscr{Z}$ carries its own braiding, tensor product, and inner homs.  The tensor product of objects $M$ and $N$ in $\mathrm{dg}\mathscr{Z}$ is the usual tensor product of complexes
\[
(M\ox N)^i=\oplus_{i_1+i_2=i} M^{i_1}\ox N^{i_2}
\]
with differential
\[
d_{M\ox N}(m\ox n)=d_M(m)\ox n+(-1)^{|m|}m\ox d_N(n),
\]
where $m$ and $n$ are assumed to be homogeneous in the above expression.  The braiding will be given by the braiding from $\mathscr{Z}$ with the addition of the standard Koszul sign:
\[
c^{\mathrm{dg}\mathscr{Z}}_{MN}:M\ox N\to N\ox M,\ \ m\ox n\mapsto (-1)^{|m||n|}r^j\cdot n\ox r_j\cdot n.
\]
The inner homs are the expected hom complexes
\[
\big(\uHom(M,N)\big)^i=\prod_{j_2-j_1=i}\uHom(M^{j_1},N^{j_2})
\]
with differential $d_{\uHom}(f)=d_Nf-(-1)^{|f|}fd_M$.  The action is the action implied by the $\uHom$ notation, $\big(h\cdot f\big)(m)=h_1\cdot \big(f(S(h_2)\cdot m)\big)$.
\par

The fact that both $d_M$ and $d_N$ commute with the $H$-action on $M$ and $N$ respectively implies that $d_{\uHom}$ commutes with the action of $H$ as well.  One can check easily that the pairing
\[
\uHom(M,N)\ox M\to N
\]
will always be a map of complexes in $\mathrm{dg}\mathscr{Z}$.  We adopt a similar notation for $\uHom_E$.
\par

In general when doing computations with chain complexes we take all elements to be homogeneous.

\subsection{The cup product}

Recall that the bar construction 
\[
\mathscr{B}B=\cdots \to B^{\ox 3}\to B^{\ox 2}\to B\to k\to 0
\]
is a dg coalgebra with comultiplication given by separation of tensors.  For $x=b_1\ox\dots\ox b_n$ we have
\[
\Delta_{\mathscr{B}B}(x)=(x)\ox (1)+(1)\ox (x)+\sum_{i=1}^{n-1} (b_1\ox\dots\ox b_i)\ox(b_{i+1}\ox\dots\ox b_n).
\]
Indeed, with the addition of the diagonal $H$-action $\mathscr{B}B$ becomes a dg coalgebra in $\mathscr{Z}=H\mathrm{mod}$.

\begin{lemma}
For any coalgebra $\Gamma$ in $\mathrm{dg} \mathscr{Z}$ and dg algebra $\Omega$ in $\mathrm{dg} \mathscr{Z}$ the inner hom complex $\uHom(\Gamma,\Omega)$ becomes an algebra in $\mathrm{dg}\mathscr{Z}$ under the (braided) convolution product
\begin{equation}\label{eq:cupprod}
f\cup g(\gamma)=(-1)^{|\gamma_1||g|}f(r^j\cdot \gamma_1)\big(r_j\cdot g\big)(\gamma_2),
\end{equation}
for $f,g\in\uHom(\Gamma,\Omega)$, $\gamma\in\Gamma$.
\end{lemma}

\begin{proof}
The verification that the product is compatible with the differential is straightforward and will be omitted.  To see that the product is a map in $\mathscr{Z}=H\mathrm{mod}$, note that we arrive at the element $\pm f(r^j\cdot \gamma_1)\big(r_j\cdot g\big)(\gamma_2)$ via the sequence of maps in $\mathscr{Z}$
$$
\ba{l}
\uHom^{\ox 2}\ox \Gamma\overset{id\ox\Delta}\longrightarrow  \uHom^{\ox 2}\ox \Gamma\ox \Gamma\overset{id\ox c\ox id}\longrightarrow \uHom\ox \Gamma\ox \uHom\ox \Gamma
\\
\overset{\mathrm{pair}\ox\mathrm{pair}}\longrightarrow \Omega\ox \Omega\overset{\mathrm{mult}}\longrightarrow \Omega,
\ea
$$
where $\uHom=\uHom(\Gamma,\Omega)$.  So, $H$-linearity tells us that for any $h\in H$, $f,g\in \uHom(\Gamma,\Omega)$, and $\gamma\in \Gamma$, we have
$$
\pm h\cdot \big(f(r^j\cdot \gamma_1)\big(r_j\cdot g\big)(\gamma_2)\big)=\pm\big(h_1\cdot f)(r^j h_3\cdot \gamma_1)\big(r_j h_2\cdot g\big)(h_4\cdot \gamma_2)
$$
Replace $h\ox \gamma_1\ox \gamma_2$ with $h_1\ox S(h_3)\gamma_1\ox S(h_2)\gamma_2=h_1\ox \Delta(S(h_2)\gamma)$ to get then
$$
\ba{rl}
\big(h\cdot (f\cup g)\big)(\gamma)&=\pm h_1\cdot \big(f(r^jS(h_3)\cdot \gamma_1)\big(r_j\cdot g\big)(S(h_2)\cdot \gamma_2)\big)\\
&=\pm\big(h_1\cdot f\big)(r^j\cdot \gamma_1)\big(r_j h_2\cdot g\big)(\gamma_2)\\
&=\big((h_1\cdot f)\cup (h_2\cdot g))\big)(\gamma).
\ea
$$
This verifies that $\uHom(\Gamma,\Omega)$ with the proposed product is a, possibly non-associative, algebra in $\mathrm{dg}\mathscr{Z}$.
\par

For associativity, take $a,b,c\in \uHom(\Gamma,\Omega)$.  We have
$$
\ba{rll}
\big((a\cup b)\cup c\big)(\gamma)&
=a(r^\nu r^j_1\cdot \gamma_1)\big(r_\nu \cdot b\big)(r^j_2\cdot \gamma_2)\big(r_j\cdot c\big)(\gamma_3)\\
&=a(r^\nu r^j\cdot \gamma_1)\big( r_\nu\cdot b)(r^l\cdot \gamma_2)\big(r_lr_j\cdot c\big)(\gamma_3) &\text{(braid relation).}\\
&
\ea
$$
On the other hand we have
$$
\ba{rll}
\big(a\cup (b\cup c)\big)(\gamma)&
=a(r^\nu\cdot \gamma_1)\big((r_\nu)_1\cdot b\big)(r^l\cdot \gamma_2)\big(r_l(r_\nu)_2\cdot c\big)(\gamma_3)\\
&=a(r^\nu r^j\cdot \gamma_1)\big(r_\nu\cdot b\big)(r^l\cdot \gamma_2)\big(r_l r_j\cdot c\big)(\gamma_3) &\text{(braid relation).}\\
\ea
$$
These expressions are exactly the same, and associativity is verified.
\end{proof}

We have an identification of graded objects in $\mathscr{Z}$
\[
C^\bullet_c(B)=\uHom(\mathscr{B}B,B)
\]
and so we get a natural product on $C^\bullet_c(B)$ induced by the convolution product on $\uHom(\mathscr{B}B,B)$.  This gives the braided Hochschild complex the structure of an associative algebra in $\mathscr{Z}$, although we've yet to establish compatibility with the differential.  In keeping with tradition, we call this product on $C^\bullet_c(B)$ the {\it (braided) cup product}.

\begin{proposition}\label{prop:dga}
The braided Hochschild cochain complex $C^\bullet_c(B)$, along with the cup product, is an algebra in $\mathrm{dg} \mathscr{Z}$.  Furthermore, for any subalgebra $E\subset B$ in $\mathscr{Z}$ the subcomplex $C^\bullet_{c,E}(B)\subset C^\bullet_c(B)$ is a subalgebra in $\mathrm{dg} \mathscr{Z}$.
\end{proposition}

By an algebra in $\mathrm{dg}\mathscr{Z}$ we mean in particular that it is a dg algebra.  In order to prove the proposition it will be helpful to have the following lemma.

\begin{lemma}\label{lem:688}
The differential on the complex $C^\bullet_c(B)=(\uHom(\mathscr{B}B,B),d_c)$ is the sum
\[
d_c=d_{\uHom(\mathscr{B}B,B)}-[\pi,-],
\]
where $\pi$ is the degree $1$ map $\mathscr{B}B\to B$ which restricts to the identity on $\mathscr{B}B^{-1}=B$, and $[\pi,-]$ is the usual graded commutator $[\pi,f]=\pi\cup f-(-1)^{|f|}f\cup\pi$.
\end{lemma}

\begin{proof}
Since $\pi$ is a map in $\mathscr{Z}=H\mathrm{mod}$ we will have $h\cdot \pi=\epsilon(h)\pi$ for each $h\in H$.  Consider now a degree $n$ function $f$ and $x=b_1\ox\dots\ox b_{n+1}$.  Then
\[
\ba{l}
(-1)^{|f|}[\pi,f](x)\\
=\pi(r^j\cdot b_1)\big(r_j\cdot f\big)(b_2\ox\dots b_{n+1})-(-1)^{n}f(b_1\ox\dots\ox b_n)\pi(b_{n+1})\\
=(r^j\cdot b_1)\big(r_j\cdot f\big)(b_2\ox\dots b_{n+1})-(-1)^{n}f(b_1\ox\dots\ox b_n)b_{n+1}\\
=(-1)^{|f|+1}d_c(f)(x)-f\big(d_{\mathscr{B}B}(x)\big)\\
=(-1)^{|f|+1}\big(d_c(f)-d_{\uHom}(f)\big)(x).
\ea
\]
This implies
\[
-[\pi,f]=d_c(f)-d_{\uHom}(f)\Rightarrow d_c(f)=d_{\uHom}(f)-[\pi,f].
\]
\end{proof}

\begin{proof}[Proof of Proposition \ref{prop:dga}]
As noted above, $\pi$ is $H$-invariant so that, by the relation $(1\ox \epsilon)(R_{21})=1\ox 1$ the cup product of $\pi$ with itself will be given by the non-braided formula
\[
\big(\pi\cup\pi\big)(x)=\pi(r^j\cdot x_1)\big(r_j\cdot \pi\big)(x_2)=\pi(x_1)\pi(x_2).
\]
Whence we see that on degree $-2$ elements we have $\big(\pi\cup\pi\big)(b\ox b')=bb'$ and find that $\pi$ solves the Maurer-Cartan equation
\[
d_{\uHom}(\pi)-\pi\cup\pi=0
\]
in the dg algebra $\uHom(\mathscr{B}B,B)$.  It is then standard that the new object
\[
(\uHom(\mathscr{B}B,B),d_{\uHom}-[\pi,-])=C^\bullet_c(B)
\]
is a dg algebra under the same product as $\uHom(\mathscr{B}B,B)$~\cite{LV}.
\par

As for the fact that $C^\bullet_{c,E}(B)$ is a subalgebra in $C^\bullet_{c}(B)$, we already know that it is a subcomplex in $\mathrm{dg} \mathscr{Z}$.  So we need only see that it is a subalgebra under the cup product.  We have for $w\in E$, $x\in \mathscr{B}B$, and $f,g\in C^\bullet_{c,E}(B)$,
\[
\ba{ll}
(f\cup g)(wx)&=f(r^j\cdot (wx_1))\big(r_j\cdot g\big)(x_2)\\
&=(r^kr^j\cdot w)\big(r_k\cdot f\big)(r^l\cdot x_1)\big(r_lr_j\cdot g\big)(x_2)\\
&=(r^j\cdot w)\big((r_j)_1\cdot f\big)(r^l\cdot x_1)\big(r_l(r_j)_2\cdot g\big)(x_2)\\
&=(r^j\cdot w)\big(((r_j)_1\cdot f)\cup((r_j)_2\cdot g)\big)(x)\\
&=(r_j\cdot w)\big(r_j\cdot (f\cup g)\big)(x).
\ea
\]
One checks by a similar application of (\ref{eq:relations}) that 
\[
f(r^j\cdot x_1)\big(r_j\cdot g\big)(wx_2)=f\big(r^j\cdot (x_1w)\big)(r_j\cdot g)(x_2),
\]
and right $E$-linearity $f\cup g(xw)=f\cup g(x)w$ follows by right $E$-linearity of $g$.  One can check easily from these three relations that $f\cup g$ satisfies all the relations necessary so that it lay in the subcomplex $\uHom_E(\mathscr{B}_EB,B)=C^\bullet_{c,E}(B)\subset C^\bullet_c(B)$.
\end{proof}

Taking cohomology then implies, from Proposition \ref{prop:dga},

\begin{corollary}\label{cor:1121}
\begin{enumerate}
\item The braided Hochschild cohomology $H^\bullet_c(B)$, as well as its relative versions $H^\bullet_{c,E}(B)$, are all algebras is $\mathrm{dg} \mathscr{Z}$ (with vanishing differential).
\item There are canonical algebra maps $H^\bullet_{c,E}(B)\to H^\bullet_c(B)$ in $\mathrm{dg} \mathscr{Z}$.
\item When $E$ is semisimple and cosemisimple, $\mathscr{Z}=YD^E_E$, and $B=A\ast E$, the algebra map $H^\bullet_{c,E}(B)\to H^\bullet_c(B)$ is an algebra isomorphism in $\mathrm{dg}\mathscr{Z}$.
\end{enumerate}
\end{corollary}

\begin{proof}
(1) and (2) are consequences of Proposition \ref{prop:dga}.  For (3) we consider also Theorem \ref{thm:relVreg}.
\end{proof}

Since $\mathscr{B}A=(\mathscr{B}B)^{coE}$ when $B=A\ast E$, we get that the braided cup product of functions $f\cup g$ restricted to $\mathscr{B}A$ recovers the standard cup product on the Hochschild cochain complex, under the identification $C^\bullet_{c,E}(B)=C^\bullet(A,B)$ of Proposition \ref{prop:relVHH}.  Whence we also get

\begin{corollary}\label{cor:749}
When $B=A\ast E$
\begin{enumerate}
\item there identification $HH^\bullet(A,B)=H_{c,E}^\bullet(B)$ of Proposition \ref{prop:relVHH} is one of graded algebras.
\item The standard Hochschild cohomology $HH^\bullet(A,B)$ is an algebra in $\mathrm{dg} YD^E_E$.
\item There is a multiplicative spectral sequence
\[
\Ext_{\mathrm{mod}E}(k,H^\bullet_{c,E}(B))\Rightarrow HH^\bullet(B).
\]
\end{enumerate}
When, additionally, $E$ is semisimple and cosemisimple
\begin{enumerate}
\item[(4)] there is an algebra identification $HH^\bullet(A,B)=H^\bullet_c(B)$ in $\mathrm{dg} YD^E_E$.
\item[(5)] There is an algebra identification $HH^\bullet(B)=H^\bullet_c(B)^E$.
\end{enumerate}
\end{corollary}

\begin{proof}
Statement (1) is clear from the discussion preceding this corollary.  (2) follows from (1).  (3) follows from the preexisting spectral sequence
\[
\Ext_{\mathrm{mod}E}(k,HH^\bullet(A,B))\Rightarrow HH^\bullet(B)
\]
from~\cite{N2}.  (4) follows from (1) and Corollary \ref{cor:1121}.  (5) follows from (3), as the spectral sequence will collapse.  (5) can also be deduced from~\cite{gerstenhaberschack86}.
\end{proof}

We can also twist by a dual cocycle to get

\begin{corollary}\label{cor:cocycprod}
Take $B=A\ast E$.  In the presence of an invertible dual cocycle $J\in E^\ast\ox E^\ast$ the algebra $HH^\bullet(A,B_J)$ is an algebra in the braided category $\mathrm{dg}D^J\mathrm{mod}$.
\end{corollary}

\begin{proof}
The result follows from Proposition \ref{prop:w_cocycle} and the fact that the $J$-twist sends dg algebras in $D\mathrm{mod}$ to dg algebras in $D^J\mathrm{mod}$~\cite{montgomery04}.
\end{proof}

We presumably will not get an identification of algebras $HH^\bullet(B)=H^\bullet_{c,E}(B)^E$ if we include a dual cocycle $J$, since $E$ is not a Hopf subalgebra of the twisted double in general, and so $C^\bullet_{c,E_J}(B_J)$ will not be a dg algebra in $E\mathrm{mod}$.  So it makes no sense to think about the ``$E$-invariant subalgebra" here.  This brings us back to the concerns expressed at the conclusion of Section \ref{sect:w_cocycle}.

\begin{remark}
Although we have ignored the issue, our production of a product on the cohomology $H^\bullet_c(B)$ is somewhat curious, since the identification
\[
H^\bullet_c(B)=\Ext_{\mathrm{mod}B^{\unl{e}}}(B,B)
\]
tells us that this object already has a natural product (the Yoneda product).  This curiosity also appears in the non-braided setting with the production of the usual cup product.  We expect that, as in the non-braided setting, the two products will agree.  Namely, we should be able to produce a dg algebra quasi-isomorphism
$$
C^\bullet_c(B)\to \End_{\mathrm{mod}B^{\unl{e}}}(\mathscr{B}ar^cB,\mathscr{B}ar^cB)
$$
just as in~\cite{N}.  One should simply apply the braiding where appropriate.
\end{remark}

\subsection{Braided commutativity}

\begin{definition}
We define the circle operation
\[
\circ:C^\bullet_{c}(B)\ox C^\bullet_{c}(B)\to C^\bullet_{c}(B)
\]
by
\[
f\circ g(x)=(-1)^{|x_1|(|g|-1)}f\big((r^j\cdot x_1)\ox \big(r_j\cdot g\big)(x_2)\ox x_3\big)
\]
and the \emph{na\"ive bracket} by
\[
[f,g]_\circ=f\circ g-(-1)^{(|f|-1)(|g|-1)}(r^j\cdot g)\circ (r_j\cdot f).
\]
\end{definition}

Here we've used the identifications $B^{\ox n_1}\ox B^{\ox n_2}\ox B^{\ox n_3}=B^{\ox n_1+n_2+n_3}$, and subsequent concatenation map $\mathscr{B}B\ox\mathscr{B}B\ox\mathscr{B}B\to \mathscr{B}B$ to view the elements $r^j\cdot x_1\ox (r_j\cdot g)(x_2)\ox x_3$ as laying in the bar construction of $B$.  One subtlety of this point is that, when $x_1$ is in $B^{\ox 0}=k$, for example, the element $r^j\cdot x_1\ox (r_j\cdot g)(x_2)\ox x_3$ is identified with $g(x_1)\ox x_2$.  Save for the presence of $R_{21}$, this formula looks exactly like the usual circle operation~\cite{gerstenhaber63}, and can be identified with composition of braided coderivations as in~\cite{stasheff}.
\par

The bracket $[,]_\circ$ will {\it not} preserve cocycles or coboundaries in general.  One needs to do a bit more work in order to produce a well behaved bracket operation which extends the Gerstenhaber bracket to the braided setting (see~\cite[Ch. 3]{negronthesis}).  This poorly behaved bracket will, however, still prove quite useful.

By the usual computations one checks

\begin{lemma}
For any subalgebra $E\subset B$ the dg subalgebra $C^\bullet_{c,E}(B)\subset C^\bullet_c(B)$ is preserved by the circle operation and na\"ive bracket.
\end{lemma}

We can now prove an essential technical result.

\begin{proposition}\label{prop:bdcup}
For $f,g$ in $C^\bullet_{c}(B)$ or $C^\bullet_{c,E}(B)$ we have
\begin{equation}\label{eq:4973}
\ba{l}
(-1)^{|f|+1}d_c(f\circ g)+(-1)^{|f|} d_c(f)\circ g-f\circ d_c(g)\\
=f\cup g-(-1)^{|f||g|}(r^j\cdot g)\cup(r_j\cdot f).
\ea
\end{equation}
In particular, if $f$ and $g$ are cocycles then
\begin{equation}\label{eq:4978}
(-1)^{|f|+1}d(f\circ g)=f\cup g-(-1)^{|f||g|}(r^j\cdot g)\cup (r_j\cdot f).
\end{equation}
\end{proposition}

\begin{proof}
This is just as in~\cite{gerstenhaber63}.  Let $\mu:\mathscr{B}_EB\to B$ be the degree $2$ function given on $B^{\ox_E 2}$ by the multiplication $b\ox b'\mapsto bb'$, and $0$ on all other tensor powers.  Then the differential $d_c$ is exactly the na\"ive bracket $d_c(f)=[\mu,f]_\circ$.  With this interpretation one checks the identity
\[
\ba{l}
[\mu,f\circ g]_\circ\\
=[\mu,f]_\circ\circ g-(-1)^{|f||g|+|f|-1}(r_j\cdot g)\cup (r_j\cdot f)-(-1)^{|f|}f\cup g-(-1)^{|f|}f\circ [\mu,g]_\circ.
\ea
\]
Rearranging gives
\[
\ba{l}
[\mu,f\circ g]_\circ-[\mu,f]_\circ\circ g+(-1)^{|f|}f\circ [\mu,g]_\circ\\
=(-1)^{|f|+1}f\circ [\mu,g]_\circ-(-1)^{|f||g|+|f|-1}r_jg\cup r_jf.
\ea
\]
Multiplying by $(-1)^{|f|+1}$ and replacing $[\mu,-]$ with $d_c$ gives the formula (\ref{eq:4973}).
\end{proof}

We call an algebra $\Omega$ in $\mathrm{dg}\mathscr{Z}$ {\it commutative in} $\mathrm{dg}\mathscr{Z}$, or {\it braided commutative}, if for each homogeneous $\omega,\omega'\in \Omega$ we have
\begin{equation}\label{eq:874}
\omega\omega'-(-1)^{|\omega||\omega'|}(r^j\cdot \omega')(r_j\cdot \omega)=0.
\end{equation}
This notion is not new~\cite{baez94}.

\begin{theorem}\label{thm:braided_comm}
The braided cohomologies $H^\bullet_c(B)$, as well as the relative versions $H^\bullet_{c,E}(B)$, are commutative algebras in $\mathrm{dg}\mathscr{Z}$.
\end{theorem}

\begin{proof}
Equation (\ref{eq:4978}) of Proposition \ref{prop:bdcup} implies that the braided commutator of any two cocycles is a coboundary.
\end{proof}

\begin{remark}
One can compare Theorem \ref{thm:braided_comm}, and its proof, to \cite[Corollary 3.13]{mastnaketal10}.  In our language, the ``Hochschild cohomology" considered in \cite{mastnaketal10} should be the braided Hochschild cohomology with coefficients in the unit $\mathbf{1}_\mathscr{Z}$.
\end{remark}

The first portion of the next result was deduced independently in \cite{schedlerwitherspoon} in the case of a finite group acting in characteristic $0$.

\begin{corollary}\label{cor:HH_comm}
For $B=A\ast E$ the cohomology $HH^\bullet(A,B)$ is braided commutative in $\mathrm{dg}YD^E_E$.  For any invertible dual cocycle $J$, the cohomology $HH^\bullet(A,B_J)$ is braided commutative in $\mathrm{dg}D^J$.
\end{corollary}

\begin{proof}
The result for $HH^\bullet(A,B)$ follows from the algebra identification $HH^\bullet(A,B)=H^\bullet_{c,E}(B)$ of Corollary \ref{cor:749}.  In the presence of the twist $J$, the result follows from Proposition \ref{prop:w_cocycle} and the fact that the twist functor $(-)_J$ sends commutative algebras in $\mathrm{dg}D\mathrm{mod}$ to commutative algebras in $\mathrm{dg}D^J\mathrm{mod}$.
\end{proof}

The following corollary seems to be new.  We let $\mathrm{Center}(\Omega)$ denote the usual graded center of a dg algebra $\Omega$.

\begin{corollary}\label{cor:in_center}
When $E$ is semisimple and $B=A\ast E$, there is an algebra inclusion
\[
HH^\bullet(B)\subset\mathrm{Center}\big(HH^\bullet(A,B)\big).
\]
\end{corollary}

\begin{proof}
Recall $HH^\bullet(B)=HH^\bullet(A,B)^E$.  Since $HH^\bullet(A,B)$ is braided commutative, we have $\bar{f}\cup \bar{g}=\pm (r^j\cdot \bar{g})\cup (r_j\cdot \bar{f})$ for any classes $\bar{f},\bar{g}\in HH^\bullet(A,B)$.  Thus, when $\bar{f}$ is invariant we have
\[
\bar{f}\cup \bar{g}=\pm(r^j\cdot \bar{g})\cup (r_j\cdot \bar{f})=\pm(r^j\cdot \bar{g})\cup (\epsilon(r_j)\bar{f})=\pm\bar{g}\cup \bar{f}.
\]
So invariant functions are central and we get the proposed embedding, since $HH^\bullet(B)$ is identified with the invariants in $HH^\bullet(A,B)=HH^\bullet_c(B)$.
\end{proof}

In general, we don't expect $HH^\bullet(B)$ to be the entire center.  For example, in the case of a finite group $G$ acting faithfully by linear automorphisms on the polynomials ring $S(V)$ generated by a vector space $V$, the cohomology $HH^\bullet(S(V),S(V)\ast G)$ in degree $0$ is just $S(V)$.  One can then check that all the elements in $HH^0(S(V),S(V)\ast G)=S(V)$ are already central in $HH^\bullet(S(V),S(V)\ast G)$, despite the fact that $S(V)^G$ is not equal to $S(V)$.  One can also see this from Proposition \ref{prop:whatever!} below.

\section{Consequences for smash products with group algebras}
\label{sect:consequences}

Fix a finite group $G$ acting on an algebra $A$.  We consider only the smash product $A\ast G$, although one could just as easily consider a crossed product $A\ast_\alpha G$, or $J$-twist, according to Proposition \ref{prop:w_cocycle}.  As has been our general convention, we let $G^{op}\subset D$ act on the left of Yetter-Drinfeld modules (opposed to $G$ acting on the right) as described in \ref{sect:smashprud}.
\par

Recall that we have the $A$-bimodule decomposition $A\ast G=\oplus_{g\in G} Ag$.  This gives a canonical decomposition
\[
H^\bullet_{c,kG}(A\ast G)=HH^\bullet(A,A\ast G)=\bigoplus_{g\in G} HH^\bullet(A,Ag),
\]
which, as we've mentioned before, expresses the canonical $G$-grading ($G$-coaction) on the dg Yetter-Drinfeld module $H^\bullet_{c,kG}(A\ast G)$.  Note that the usual Hochschild cohomology $HH^\bullet(A,A)= HH^\bullet(A,Ae)$ sits inside the cohomology $HH^\bullet(A,A\ast G)$, and hence $H^\bullet_{c,kG}(A\ast G)$, as the graded subalgebra of coinvariants.  The cohomology $HH^\bullet(A)$ then acts naturally on each component $HH^\bullet(A,Ag)$ by the cup product.

\begin{proposition}\label{prop:whatever!}
Suppose $G$ is a finite group and, for each $g\in G$, let $I_g$ be the ideal in $HH^\bullet(A)$ generated by all the classes $(1-g)\cdot Y$, with $Y\in HH^\bullet(A)$.  Each $HH^\bullet(A)$-module $HH^\bullet(A,Ag)$ is annihilated by $I_g$, both on the left and the right.  Furthermore, each $HH^\bullet(A,Ag)$ is a {\it central} bimodule over $HH^\bullet(A)/I_g$ and $HH^\bullet(A)$.
\end{proposition}

\begin{proof}
For any $Y_e\in HH^\bullet(A)$ and $Y_g\in HH^\bullet(A,Ag)$ braided commutativity gives
\[
Y_eY_g=Y_g(Y_e\cdot g)=\big(g\cdot Y_e\big)Y_g\Rightarrow \big((1-g)\cdot Y_e\big)Y_g=0.
\]
Hence the cohomology is annihilated by each $(1-g)\cdot Y_e$ and therefore also the ideal $I_g$.  A similar computation holds on the right as well.  The fact that $HH^\bullet(A)/I_g$ acts centrally on $HH^\bullet(A,Ag)$ follows from the fact that the induced $g$-action on the quotient $HH^\bullet(A)/I_g$ is trivial.  Now we see that $HH^\bullet(A)$ also acts centrally, as the right and left actions factor through the quotient $HH^\bullet(A)/I_g$.
\end{proof}

The ideals $I_g$ will be quite large in general.  In Section \ref{sect:ideals} we explain what work these ideals do in a geometric setting.
\par

Although we leave our investigation here, we remark that braided commutativity will also produce certain generic relations for the Hochschild cohomology of other smash products $A\ast E$, e.g. when $E$ is a restricted enveloping algebra.  These relations are not quite as straightforward to analyze as in the case of a group ring.

\subsection{Understanding the ideals $I_g$ from a geometric perspective}\label{sect:ideals}

Suppose $\mathrm{char}(k)=0$ and let $X$ be a smooth affine scheme over $k$.  Fix also a finite group $G$ acting on $X$.  We let $k[X]$ denote the algebra of global functions on $X$.  In this case the HKR theorem tells us that
$$
HH^\bullet(k[X])=\Wedge^\bullet_{k[X]}T_X,
$$
where $T_X$ is the module of global vector fields on $X$ (or, derivations on $k[X]$).  Take $X^g$ to be the fixed subscheme for a given $g\in G$.
\par

In this case each ideal $I_g$ will be generated in degrees $0$ and $1$.  We will have
$$
HH^0(k[X])/I_g^0=k[X]/I_g^0=k[X]/(f-{^gf})_{f\in k[X]}=k[X^g],
$$
and for the whole cohomology there is a canonical algebra identification
$$
HH^\bullet(k[X])/I_g=\wedge^\bullet_{k[X^g]}(T_X)|X^g/(Y-g\cdot Y)_{|Y|=1}=\Wedge^\bullet_{k[X^g]}T_{X^g}.
$$
Proposition \ref{prop:whatever!} then tells us that the algebra of polyvector fields for each fixed space $X^g$ will act naturally on the cohomology $HH^\bullet(k[X],k[X]g)$.  Each of these actions is defined canonically and globally.
\par

In the case in which $X=\mathbb{A}^n_k$ and $G$ acts by linear automorphisms we understand that
\begin{enumerate}
\item[(a)] the ideal $I_g$ is the entire annihilator of each $HH^\bullet(k[X],k[X]g)$ and
\item[(b)] each cohomology $HH^{\mathrm{codim}X^g}(k[X])$ is a free rank $1$ $k[X^g]$-module, let's call it $L_g$, and the restriction of the action map
$$
\bigoplus_{g\in G} \Wedge^\bullet_{k[X^g]}T_{X^g}\ox_{k[X^g]}L_g\to HH^\bullet(k[X],k[X]\ast G)=H^\bullet_c(k[X]\ast G)
$$
is a $HH^\bullet(k[X])$-bimodule isomorphism.
\end{enumerate}
This follows by the careful analysis of~\cite{SW}, for example.  Both (a) and (b) are also known to hold when we replace $k[X]$ with smooth functions on a manifold ~\cite{neumaieretal06,halbouttang10}.  In the case $X=\mathbb{A}^n_k$ the $L_g$ form their own subalgebra in the cohomology $k[X][L_g:g\in G]$ (basically the volume subalgebra in~\cite{SW}).  This subalgebra, along with the $HH^\bullet(A)$-action, which we've described in general, give the entire algebra structure on $HH^\bullet(k[X],k[X]*G)=H_c^\bullet(k[X]\ast G)$.  One then takes invariance to get the cup product on the usual Hochschild cohomology $HH^\bullet(k[X]\ast G)$.
\par

It seems that the above statements will hold for an arbitrary smooth affine $G$-scheme $X$, where we propose now that the $L_g$ are locally free.  However, the details have yet to emerge in the literature.

\section*{Acknowledgements}

Thanks to Travis Schedler and Sarah Witherspoon for agreeing to share their unpublished notes~\cite{schedlerwitherspoon}, and for helpful conversations in general.  Thanks also to Richard Ng, who has had an influence on my thinking about braided categories.

\bibliographystyle{abbrv}

\def\cprime{$'$}

\end{document}